\documentclass[12pt,reqno]{amsart}

\usepackage{fullpage}
\usepackage{microtype}
\usepackage{amsmath,amssymb,amsthm,mathrsfs}
\usepackage{mathtools}
\usepackage{tikz}
\usetikzlibrary{arrows.meta,shapes.geometric,fit,backgrounds}
\usepackage{float}

\usepackage{hyperref}
\hypersetup{colorlinks=true, linkcolor=blue, citecolor=blue, urlcolor=blue}

\theoremstyle{plain}
\newtheorem{theorem}{Theorem}[section]
\newtheorem{lemma}[theorem]{Lemma}
\newtheorem{proposition}[theorem]{Proposition}
\newtheorem{corollary}[theorem]{Corollary}

\newtheorem{conj}[theorem]{Conjecture}

\theoremstyle{definition}
\newtheorem{definition}[theorem]{Definition}
\newtheorem{remark}[theorem]{Remark}

\newcommand{\F}{\mathbb{F}}

\newcommand{\PG}{\mathrm{PG}}
\newcommand{\Z}{\mathbb{Z}}
\newcommand{\R}{\mathbb{R}}

\newcommand{\IM}{\mathrm{IM}}
\newcommand{\PL}{\mathrm{PL}}
\newcommand{\IMPH}{\mathrm{IM}_{\mathrm{PH}}}
\newcommand{\Nik}{\mathrm{Nikodym}}
\newcommand{\II}{\mathcal{I}}

\newcommand{\abs}[1]{\left|#1\right|}

\newcommand{\infnorm}[1]{\lVert #1 \rVert_\infty}

\newcommand{\hide}[1]{}

\DeclareMathOperator{\Norm}{N}

\begin{document}
\title{Large point-line matchings and small Nikodym sets}

\author{Zach Hunter}
\thanks{Department of Mathematics, ETH, Zurich, Switzerland. Email: {\tt zachtalkmath@gmail.com}.}

\author{Cosmin Pohoata}
\thanks{Department of Mathematics, Emory University, Atlanta, GA. Email: {\tt jbaverstraete@gmail.com}.}

\author{Jacques Verstraete}
\thanks{Department of Mathematics, UCSD, San Diego, CA. Email: {\tt cosmin.pohoata@emory.edu}.}
\email{}

\author{Shengtong Zhang}
\thanks{Department of Mathematics, Stanford University, Stanford, CA. Email: {\tt stzh1555@stanford.edu}.}

\date{\today}

\begin{abstract}
\vspace{+15mm}
For any integer $d \geq 2$ and prime power $q$, we construct unexpectedly large induced matchings in the point-line incidence graph of $\mathbb{F}_{q}^{d}$ by leveraging a new connection with the Furstenberg-S\'ark\"ozy problem from arithmetic combinatorics. In particular, we significantly improve the previously well-known baselines when $q$ is prime, showing that $\mathbb{F}_{q}^{2}$ contains matchings of size $q^{1.233}$ and $\mathbb{F}_{q}^{d}$ contains matchings of size $q^{d-o_{d}(1)}$. 

These results and their proofs have several applications. First, we also obtain new constructions for finite field Nikodym sets in dimension $d \geq 2$, improving recent results of Tao by polynomial factors. For example, when $q$ is prime, we show the existence of Nikodym sets in $\F_q^d$ of size $q^d - q^{d - o_d(1)}$. Second, we construct a new minimal blocking set in $\mathrm{PG}(2,q)$, solving a longstanding problem in finite geometry. Third, we obtain new constructions for the minimal distance problem (in $\mathbb{R}^{2}$ and also in higher dimensions), improving a recent result of Logunov-Zakharov.

We also obtain analogous results for general finite fields with large characteristics. In particular, in one of our constructions we introduce a new special set of points inside the norm hypersurface in $\mathbb{F}_{q}^{d}$, which directly generalizes the classical Hermitian unital and which may be of independent interest for applications. 
\end{abstract}

\maketitle

\newpage

\tableofcontents

\newpage


\section{Introduction} \label{sec:intro}

\subsection{Induced matchings and Nikodym sets}

For an integer $d \geq 1$ and prime power $q$, let $\mathbb{F}_{q}^{d}$ denote the $d$-dimensional vector space over the finite field $\mathbb{F}_{q}$ of order $q$. Let $\II_q^{(d)}$ be the bipartite graph whose left vertex set is the set of points in $\F_q^d$, whose right vertex set is the set of affine lines in $\F_q^d$, and where a point is adjacent to a line if and only if it lies on that line. The object that we will be studying in this paper is the following parameter. 

\begin{definition}[Induced matchings]\label{def:imd}
{\it{Let $\IM(d,q)$ be the maximum size of an induced matching in $\II_q^{(d)}$.}}
\end{definition}
Equivalently, $\IM(d,q)$ denotes the largest $m$ for which there exist pairs
\[
  (p_1,\ell_1),\dots,(p_m,\ell_m)
\]
with $p_i\in\ell_j$ if and only if $i=j$. We will be interested in the regime where the dimension $d$ is fixed and $q$ is large. 

The problem of estimating $\IM(2, q)$ was recently highlighted by Cohen--Pohoata--Zakharov \cite{CPZ}, in a surprising connection with the Heilbronn triangle problem. As noted in \cite[\S1.5]{CPZ}, a Szemer\'edi--Trotter type incidence theorem of Vinh
\cite{Vinh2011ST} implies the general upper bound
\begin{equation}\label{eq:vinh}
  \IM(2,q) \le q^{3/2}+q.
\end{equation}
When $q = p^2$ is a square, \eqref{eq:vinh} is sharp up to constants, because of the so-called Hermitian unitals defined on the affine plane by
\begin{equation}\label{eq:hermitian-curve}
  P:=\{(a,b)\in\F_q^2:\ a^{p+1}+b^{p+1}=1\}.
\end{equation}
It is well-known that $|P|=p^{3}-p=q^{3/2}-q^{1/2}$. Through each $(a,b)\in P$
there is a unique tangent line $\ell_{a,b}$ meeting $P$ only at $(a,b)$; in affine coordinates one may take
\begin{equation}\label{eq:hermitian-tangent}
  \ell_{a,b}=\{(a,b)+t(b^{p},-a^{p}): t\in\F_q\}
  =\{(x,y)\in\F_q^2:\ a^{p}x+b^{p}y=1\}.
\end{equation}
Thus the pairs $\{((a,b),\ell_{a,b}): (a,b) \in P\}$ form an induced matching in $\II_q$ of size $|P|$, giving
$\IM(2,q)\ge q^{3/2}-q^{1/2}$. Hermitian unitals have recently been featured in several constructions in extremal combinatorics, most notably in the recent paper of Mattheus--Verstra\"ete \cite{MV} on the asymptotics of the off-diagonal Ramsey $R(4,t)$. 

That being said, the unital construction is not applicable when $q$ is not a perfect square, and the best known lower bound for $\IM(2,q)$ prior to this work has been of the form $\IM(2,q) \gg q \log q$\footnote{Throughout the paper, we will write $f\gg g$ or $f = \Omega(g)$ or if there exists a positive constant $C$ such that $|f(x)|\geq C g(x)$ for all $x$. We will also write $f \ll g$ or $f = O(g)$ if $g \gg f$, and $f\asymp g$ if $f \gg g$ and $f \ll g$ both hold. We say $f= o(g)$ if for any $\varepsilon>0$, $|f(x)|\le \varepsilon g(x)$ for all sufficiently large $x$.}. This is achieved by either choosing the point set randomly or lifting a clique from the Paley graph (see Section~\ref{sec:prime-2d} for the important backstory).

Our foundational result improves upon this ``baseline bound" by a polynomial factor for all prime $q$. 

\begin{theorem}\label{thm:d=2}
For all primes $q$, we have
\[
  \IM(2,q) \gg q^{1.2334}.
\]
\end{theorem}

For $d\ge 3$, no analogue of \eqref{eq:vinh} is known, and it is in fact an open problem to show $\IM(d,q) = o(q^d)$. This question is typically stated in a different language, in terms of so-called {\it{Nikodym sets}}. We next formally define Nikodym sets and state their connections.

Fix $d\ge 2$ and a prime power $q$.
For $x\in\F_q^d$ and $v\in\F_q^d\setminus\{0\}$,
write
$
  \ell(x,v):=\{x+t v:\ t\in\F_q\}
$
for the affine line through $x$ with direction $v$, and
$
  \ell(x,v)^{\!*}:=\ell(x,v)\setminus\{x\}
$
for the corresponding punctured line.

\begin{definition}[Nikodym set]
A set $N\subset \F_q^d$ is a \emph{Nikodym set} if for every $x\in\F_q^d$ there exists
$v\in\F_q^d\setminus\{0\}$ such that $\ell(x,v)^{\!*}\subset N$. Let $\Nik(d, q)$ be the size of the smallest Nikodym set in $\F_q^d$.
\end{definition}

Nikodym sets are a finite-field analogue of the classical Euclidean Nikodym problem, and are closely related
(via projective transformations) to Kakeya sets. See, for instance, Dvir's beautiful survey \cite{dvir2012incidence} and the references therein.

From this perspective, point--line matchings can be regarded as a weaker version of Nikodym sets, where only points outside $N$ are required to have a punctured line contained in $N$. 

\begin{definition}[Weak Nikodym set]\label{def:weak-nik}
A set $N\subset \F_q^d$ is a \emph{weak Nikodym set} if for every $x\in \F_q^d\setminus N$
there exists $v\in\F_q^d\setminus\{0\}$ such that $\ell(x,v)^{\!*}\subset N$. 
\end{definition}
To keep a clear picture, we summarize the relationship between Nikodym sets and weak Nikodym sets in the following simple proposition. See Section~\ref{sec:Nikodym-I} for the complete proof.
\begin{proposition}
\label{prop:nikodym-basic}
1) Any Nikodym set is a weak Nikodym set, but not vice versa.

2) A set in $\F_q^d$ is a weak Nikodym set if and only if its complement is contained in an induced matching in $\II_q^{(d)}$.

3) If a set $N \subset \F_q^d$ is a weak Nikodym set, then the set $N \times \F_q \subset \F_q^{d + 1}$ is a Nikodym set.
\end{proposition}
Part 2) and 3) shows the implications $$\IM(d,q)=o(q^d) \Longrightarrow \Nik(d, q) =q^d-o(q^d) \Longrightarrow  \IM(d - 1,q) = o(q^{d - 1}).$$
The problem of showing $\Nik(d, q) =q^d-o(q^d)$ is a well-known open problem in the area, see for example \cite{LundSarafWolf2018}.

We will next establish new lower bounds for the quantity $\IM(d,q)$ when $d \geq 3$, which due to Proposition \ref{prop:nikodym-basic} will also come with new constructions for Nikodym sets in $\mathbb{F}_{q}^{d}$. 

\subsection{Large induced matchings in higher dimensions}

In higher dimensions, the ``baseline bound" is $\IM(d,q)\gg q^{d - 1}\log q$. We will discuss the history of this bound in the next subsection. Before our work, the only known improvement upon the baseline bound when $q$ is a prime is a constant factor improvement due to a recent result of Tao \cite{Tao2025} on Nikodym sets. 

Before we talk more about Nikodym sets, let us start with the simple observation that by Proposition~\ref{prop:nikodym-basic} (part 2 and 3), we can already derive the inequality
\begin{equation} \label{eq:IMlift}
    \IM(d, q) \geq q \IM(d - 1, q).
\end{equation}
Hence, for example, Theorem~\ref{thm:d=2} already gives us a polynomial improvement of $\IM(d, q) \geq q^{d - 0.7776}$ over the trivial bound when $q$ is a prime.

Our next results significantly improve upon such bounds in high dimensions. In particular, when $q$ is a prime, we show that as $d \to \infty$, the exponent can be improved to $d - o_d(1)$.
\begin{theorem}
    \label{thm:prime-high-dim}
    There exist constants $\varepsilon_d > 0$ with $\varepsilon_d \ll (\log d)^{-1}$, such that for any dimension $d$ and prime $q$, we have
    $$\IM(d, q) \gg_d q^{d - \varepsilon_d}.$$
\end{theorem}
In other words, while Theorem \ref{thm:prime-high-dim} does not refute the possibility that $\IM(d, q) = o(q^d)$ holds, it shows that a dimension--independent power saving is not possible. 

\medskip
\subsection{The story for prime powers $q$}
 The story here technically starts with the work of Sz\H{o}nyi et. al. \cite{Szonyi} who showed that for any $k \geq 2$ and $q=p^{k}$ one can use ovoids to construct a minimal blocking set in $\mathrm{PG}(2,q)$ of size $p^{k + 1} + 1$. Their construction contains only one point at infinity, and removing that point immediately gives a point--line matching of size $p^{k + 1}$, proving $\IM(2, q) \geq q^{1 + 1/k}$. When $k=2$, this recovers the lower bound $\IM(2,q) \geq q^{3/2}$ from earlier. For general $k \geq 3$, the inequality \eqref{eq:IMlift} readily gives
 \begin{equation} \label{basepower}
 \IM(d,q) \geq q^{d - 1 + 1/k}.
 \end{equation}
As a first new observation, we record that Theorem~\ref{thm:prime-high-dim} above directly implies an improvement for prime powers with small exponent compared to the dimension. 
\begin{corollary}
    \label{cor:prime-power}
    For $s \ge 2$ and $d \ge 3$, there exist constants $\varepsilon_{d, s} > 0$ with $\varepsilon_{d, s} \ll s / \log d$, such that for any prime power $q = p^s$ with $p$ a prime greater than $s$, we have
    $$\IM(d, q) \gg_{d, s} q^{d - \varepsilon_{d, s}}.$$
\end{corollary}

More substantially, we can also show a large polynomial improvement over \eqref{basepower}, for any power that is less than the dimension.
\begin{theorem}\label{thm:small-power}
For every positive integers $k$ and $d$ with $d \geq k \geq 2$, if $q=q_{0}^k$ is a prime power with prime $q_{0}$, then we have
\[
  \IM(d,q)\ \gg_{d} \ q^{d-1/k}.
\]
\end{theorem}
Instead of considering ovoids, the proof of Theorem \ref{thm:small-power} will proceed by constructing a set of points inside a norm hypersurface, which will directly generalize the Hermitian unital construction from \eqref{eq:hermitian-curve} and \eqref{eq:hermitian-tangent}. For these reasons, we believe this construction is interesting for independent reasons and should have more applications. We discuss the proofs of Corollary \ref{cor:prime-power} and of Theorem \ref{thm:small-power} in Section \ref{sec:matching-highpower-II}. 

Finally, by combining Corollary~\ref{cor:prime-power}, Theorem~\ref{thm:small-power} and a new argument, we achieve bounds of the form $q^{d - o_d(1)}$ whenever the base of $q$ is much larger than the exponent.

\begin{theorem}
    \label{thm:prime-power-high-dim}
    There exists constants $\varepsilon_d \ll (\log \log d)^{-1}$, such that if $q = p^t$ is a prime power, then
    $$\IM(d, q) \gg_d q^{d - \varepsilon_d - 2 \log t / \log p}.$$
\end{theorem}
\subsection{Small Nikodym sets}
Combining our point--line matching construction with Proposition~\ref{prop:nikodym-basic}, we can immediately generate improved constructions of Nikodym sets in dimensions $\geq 3$. Before we make this more precise, let us first review the literature on this problem. 

For $d = 2$, the best known lower and upper bounds are of the form
$$q^2 - q^{3/2} - 1 \leq \Nik(2, q) \leq \begin{cases}
    q^2 - q^{3/2} + O(q \log q), \text{$q\ $is a perfect square} \\
    q^2 - O(q \log q), \text{$q\ $ is not a perfect square}
\end{cases}.$$
For $d \geq 3$, there is a significant gap between lower and upper bounds on $\Nik(d, q)$. A sequence of works starting from Dvir's breakthrough \cite{BukhChao2021, Dvir2009Kakeya,LundSarafWolf2018} showed the general lower bound
$$\Nik(d, q) \geq \frac{q^d}{2^{d - 1}} + O(q^{d - 1}).$$
If the characteristic $q_0$ of $\F_q$ is bounded by a constant, then a much stronger lower bound is known \cite{guo2013affine}: for some $\epsilon = \epsilon(q_0, d) > 0$, we have
$$\Nik(d, q) \geq q^d - O(q^{(1 - \epsilon) d}).$$
The problem of upper bounds is highlighted in recent works \cite{Deepmind2025, Tao2025} of Tao in collaboration with Georgiev, G\'{o}mez-Serrano, Wagner and the AI tools \texttt{AlphaEvolve} and \texttt{Deep Think} operated by Google DeepMind. It is the first of 67 mathematical problems considered by the collaboration. Tao \cite{Tao2025} first noted that for any $q, d_1, d_2$ we have the ``product construction"
$$\Nik(d_1 + d_2, q) \leq \Nik(d_1, q) \Nik(d_2, q).$$
By taking the product of $2$-dimensional Nikodym sets, we have
$$\Nik(d, q) \leq
    q^d - \lfloor \frac{d}{2} \rfloor q^{d - 1/2} + O(q^{d - 1} \log q), \text{when $q$ is a perfect square}.$$
The focus of Tao's paper is on non-square $q$ with unbounded characteristics. Tao notes that a simple probabilistic argument, i.e. picking each point independently at random, shows that for every odd prime power $q$, we have
$$\Nik(d, q) \leq q^d - (d - 1 + o(1))q^{d - 1} \log q.$$
Via a combination of \texttt{AlphaEvolve} (for $q$ prime), \texttt{Deep Think}, and human guidance, Tao identified a construction based on deleting random quadratic varieties, which improves the baseline by a constant factor
$$\Nik(d, q) \leq q^d - \left(\frac{d - 2}{\log 2} + 1 +  o(1)\right)q^{d - 1} \log q.$$
As an application of our results on point--line matchings, we improve Tao's bound by polynomial factors when $q$ is a prime or most prime powers. Again, the only regime where we cannot improve Tao's bound is when $q$ is a high power of a small prime. 

\begin{theorem}
    \label{thm:nikodym-3d}
    1) For $d = 3$ and every odd prime power $q = p^t$, we have
    $$\Nik(3, q) \leq q^3 - \Omega_t(q^{2.1167})$$
    2) For sufficiently large $d \geq 3$ and $q = p^t$ a prime power, we have
    $$\Nik(d, q) \leq q^d - \Omega_d\left(q^{d - \epsilon_d - 2\log t / \log p}\right).$$
    where $\epsilon_d \ll (\log\log d)^{-1}$.
\end{theorem}


 However, as our construction requires going up one dimension, this method does not construct small Nikodym sets in dimension $2$. Nevertheless, there is an embedding trick one can use, taking advantage of additional extra properties of our high-dimensional matchings, that will allow us to also get the following.
\begin{theorem}
    \label{thm:nikodym-2d}
    There exists an absolute constant $c > 0$ such that for prime $q$, we have
    $$\Nik(2, q) \leq q^{2} - q^{1+c}.$$
\end{theorem}

\subsection{Minimal blocking sets} Let $\PG(2, q)$ denote the projective plane over $\F_q$. We say $B \subset \PG(2, q)$ is a \textit{blocking set} if for every affine line $\ell\subset \PG(2, q)$, $\ell\cap B\neq \emptyset$. We say $B$ is a minimal blocking set, if $B'$ is not a blocking set for every proper subset $B'\subset B$. Equivalently, every point
of a minimal blocking set is \emph{essential}: for each $p\in B$ there exists a \emph{tangent line}
$\ell$ through $p$ such that $\ell\cap B=\{p\}$. Blocking sets and their extremal variants have a long history
and appear throughout finite geometry, coding theory, and incidence combinatorics; see, for instance,
the survey of Blokhuis~\cite{Blokhuis1996BlockingSets}.
In particular, Bruen--Thas~\cite{BruenThas1977} showed that minimal blocking sets can have size on the
$q^{3/2}$ scale (and that this scale is sharp when $q$ is a square, as witnessed by the Hermitian unital). 

The induced-matching problem considered in this paper is in fact closely related to the ``tangent line''
characterization above: the affine part of a minimal blocking set automatically supports an induced matching. In a remarkable paper \cite{Szonyi1992}, Sz\H{o}nyi~ constructed large minimal blocking sets in $\mathrm{PG}(2,q)$ using (what can be rephrased as) independent sets in Paley-type
Cayley graphs. This construction will be the starting point of our proof of Theorem \ref{thm:d=2}. In Section \ref{sec:prime-2d}, we will start by isolating the induced-matching aspect of this lift, and then observe that using dense sets without square differences in the integers rather than Paley graphs leads to significantly larger point-line induced matchings. This will be the content of Theorem \ref{thm:d=2}. Subsequently, we will generalize this new lift to higher dimensions in Theorems \ref{thm:prime-high-dim} and \ref{thm:prime-power-high-dim}. 

Before we get to that, in this section we would like to record the following simple connection between minimal blocking sets and Nikodym sets in $\mathbb{F}_{q}^{2}$.

\begin{proposition} \label{prop:blocking}
Let $B(2,q)$ denote the maximum cardinality of a minimal blocking set in $\PG(2, q)$. Then,
 $$B(2,q)\ge q^2-\Nik(2,q).$$
\end{proposition}
Combining this with Theorem~\ref{thm:nikodym-2d}, we settle a longstanding problem in finite geometry: We construct a polynomially larger blocking set in $\operatorname{PG}(2,q)$ when $q$ is a prime.

\begin{theorem} \label{thm:blocking-2d-poly}
There exists an absolute constant $c > 0$ such that the following holds: for every prime $q$, there exists a minimal blocking set in $\operatorname{PG}(2,q)$ of size at least $\Omega(q^{1+c})$, i.e.
    $$B(2,q) \gg q^{1+c}.$$ 
\end{theorem}

\subsection{The minimal distance problem}

Determining $\IM(d,q)$ may be viewed as a finite-field analogue for Euclidean point--line separation problems. In particular, in their most recent work on the Heilbronn triangle problem \cite{CPZ}, Cohen, Pohoata and Zakharov introduced the following \emph{minimal distance problem}:
given pairs $\{p_i\in \ell_i\}_{i=1}^n$ with $p_i\in[0,1]^2$ and $\ell_i$ a line through $p_i$, what is the maximum value of
\[
  \delta := \min_{i\neq j} d(p_i,\ell_j),
\]
where $d(\cdot,\cdot)$ denotes Euclidean distance. Observe that this can be easily rephrased induced matching problem between points and $\delta$--tubes. In \cite{CPZ}, they showed that necessarily $\delta\lesssim n^{-2/3+o(1)}$,
which in turn implied the current record of $\Delta(n)\lesssim n^{-7/6+o(1)}$ for the Heilbronn triangle problem
($\Delta(n)$ being the smallest triangle area determined by $n$ points in the unit square); see \cite{CPZ} and the discussion therein.

In \cite{maldague2025heilbronn}, Maldague, Wang and Zakharov also recently proposed higher-dimensional analogues of the minimal distance problem for point--line pairs, which we denote by $\PL_d(\gamma)$. 
\begin{definition}
    Let $\PL_d(\gamma)$ be the following proposition: for every large $n$ and any points $p_1, \cdots, p_n \in [0, 1]^d$ and lines $\ell_1, \cdots, \ell_n$ with $p_i \in \ell_i$, there exists $i \neq j$ with 
    $$d_\infty(p_i, \ell_j) \leq n^{-\frac{1}{d-\gamma} + o(1)}.$$
    Here the $d_{\infty}$--distance between point and line is defined as
    $$d_\infty(p, \ell) := \min_{q \in \ell} \lVert{p - q \rVert}_\infty.$$
\end{definition}
This choice of norm is convenient for product-type constructions and is standard in this context (recall that all norms are equivalent up to dimension-dependent constants). We also note that the statement $\PL_d(0)$ is trivially true in any dimension $d$: it simply reflects the fact that among any $n$ points inside $[0,1]^d$ there always exist two points within distance $2n^{-1/d}$ from each other (a simple consequence of the pigeonhole principle). 

In dimension $2$, the main result from \cite{CPZ} established precisely the validity of $\PL_2(1/2)$. In \cite{maldague2025heilbronn}, Maldague, Wang, and Zakharov recently extended the methods from \cite{CPZ0} and \cite{CPZ} and also established the $3$--dimensional analogue $\PL_3(\gamma_0)$, for some absolute constant $\gamma_0>0$. In turn, this implied the first polynomial improvement for the Heilbronn triangle problem in three dimensions. In dimension $d \geq 4$, determining whether $\PL_d(\gamma)$ holds for any $\gamma > 0$ is an open problem with connection to higher--dimensional Heilbronn triangle problems.

In this paper, we find new constructions that show $\PL_d(\gamma)$ cannot hold for certain $\gamma$. Prior to this work, the best result in this direction comes from a recent elegant construction of Logunov and Zakharov \cite{LogunovZakharov2025Fractal}. In this paper, they introduce a self-affine (``fractal-like'') family of point--line pairs $(p_i, \ell_i)$ in $[0,1]^2$ with
\begin{equation} \label{LZ} 
  d_{\infty}(p_i, \ell_j)\ \gg\ n^{\eta -1}, \quad \forall i \neq j
\end{equation}
for some absolute, yet inexplicit, constant $\eta >0$. This translates to the fact that $\PL_2(1-\eta')$ fails for any constant $\eta' \in (0, \frac{\eta}{1 - \eta})$. We note that
\begin{equation} \label{eq:PLrelation}
  \PL_d(\gamma)\ \Longrightarrow\ \PL_{d-1}(\gamma),
\end{equation}
since one can embed a configuration in $[0, 1]^{d - 1}$ into a coordinate hyperplane in $[0,1]^d$ and take well--separated translates in the new coordinate. Hence \eqref{LZ} also implies that $\PL_d(1 - \eta')$ fails in every dimension $d \geq 2$.

\medskip

Our contribution is a general mechanism for producing well-separated Euclidean point--line configurations from
large \emph{integer} point--line matchings with \emph{bounded slopes}.
\begin{theorem}
\label{thm:lattice-to-PL}
Fix $d\ge 1$ and let $N,M,L\ge 1$ be integers with $N\ge ML$.
Let $P\subset [N]^d\times [M]\subset \mathbb{Z}^{d+1}$ be a set of lattice points.
Assume that for each $p\in P$ we are given an integer direction vector
\[
  s_p=(u_p,1)\in \mathbb{Z}^{d+1}
  \qquad\text{with}\qquad
  \|u_p\|_\infty\le L,
\]
such that for all distinct $p,p'\in P$ one has
\[
  p'\notin p+\mathbb{R}s_p
\]
which, due to the last coordinate of $s_p$ being $1$, is equivalent to $p'\notin p+\mathbb{Z}s_p$. Then there exists a set of points $p_1, \cdots, p_{|P|} \in [0, 1]^{d + 1}$ and lines $\ell_1, \cdots, \ell_{|P|}$ with $p_i \in \ell_i$ such that for each $i \neq j$, we have
\[
  d_\infty(p_i, \ell_j)\ \ge\ \frac{1}{2N}.
\]
\end{theorem}
Using this framework, we can pass from our finite-field induced matchings (after choosing convenient integer representatives)
to counterexamples for various $\PL_d(\gamma)$. In dimension $2$, using our proof of Theorem~\ref{thm:d=2}, we significantly improve on the Logunov and Zakharov's result. Furthermore, we discover a novel connection between the minimal distance problem and the Furstenberg--S\'{a}rk\"{o}zy problem.

\begin{corollary}
\label{cor:PL2}
\begin{enumerate}
\item $\PL_2(0.7666)$ is false.

\item If $\PL_2(0.5 + c)$ is true for any $c > 0$, then any square--difference--free subset of $[N]$ has size at most $N^{1 - c + o(1)}$.
\end{enumerate}
\end{corollary}

We refer to Section \ref{sec:conclusion} for a discussion of the Furstenberg--S\'{a}rk\"{o}zy problem and more context around Part 2) of Corollary \ref{cor:PL2}.

In higher dimensions, we can leverage the proof of Theorem~\ref{thm:prime-high-dim} to prove the following. 
\begin{corollary}
    \label{cor:PLd}
    For any $\gamma > 0$, there exists some $d_0=d_0(\gamma) $ such that $\PL_{d_0}(\gamma)$ fails. Quantitatively, we have $d_0(\gamma) = \exp(O(\gamma^{-1}))$.
\end{corollary}
By \eqref{eq:PLrelation}, Corollary \ref{cor:PLd} also implies that $\PL_d(\gamma)$ fails for every $d>d_0(\gamma)$.

\smallskip

\noindent {\bf{Organization.}} The diagram below indicates the organization for the remainder of the paper, with arrows indicating logical dependencies. For readers interested in applications, only the prime cases (Sections~\ref{sec:prime-2d} and \ref{sec:prime-highdim}) are necessary.

\medskip

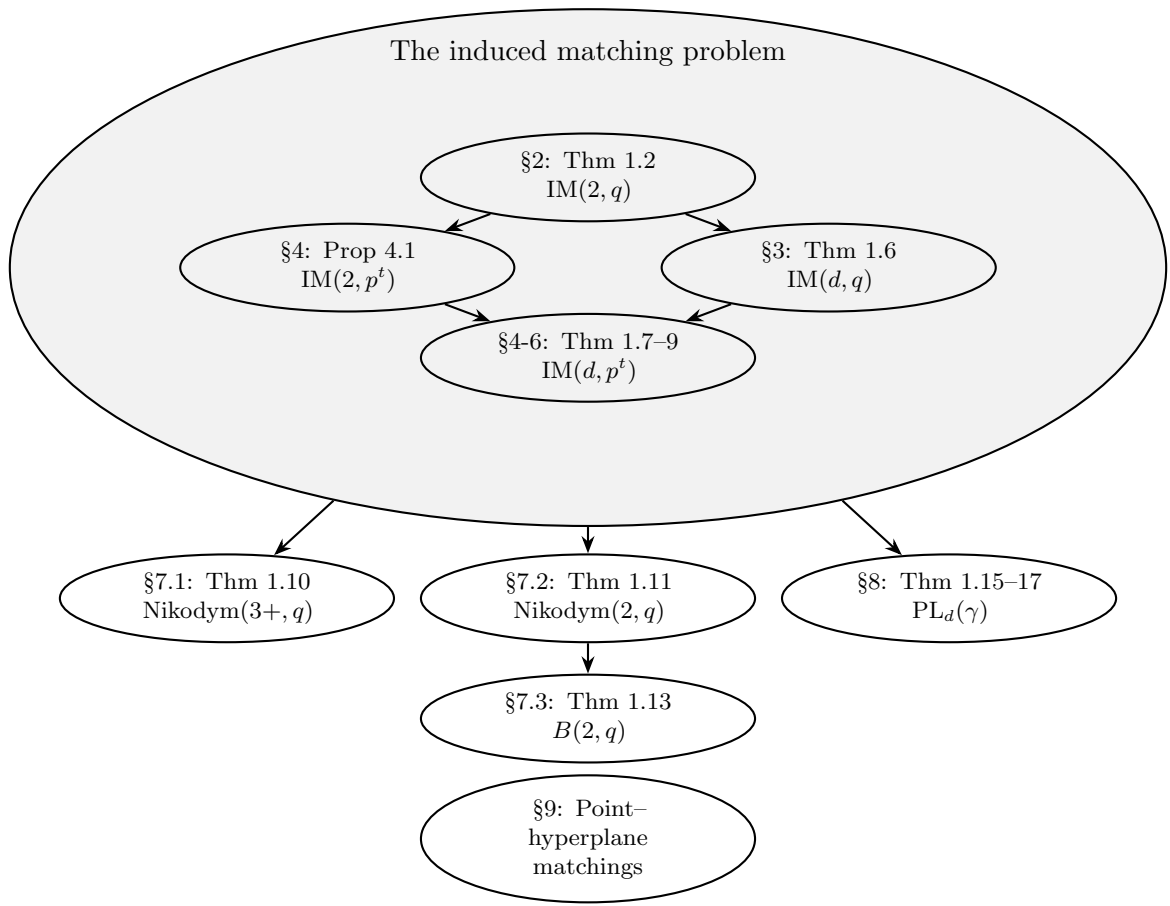
\begin{figure}[H]
\centering
\begingroup
\hypersetup{draft}
\begin{tikzpicture}[scale=0.8,
  bubble/.style={
    draw, thick, ellipse, align=center, font=\scriptsize,
    text width=3.0cm, minimum height=1.05cm, inner sep=2pt
  },
  cluster/.style={
    draw, thick, ellipse, fill=gray!10,     inner xsep=0pt, inner ysep=18pt   
  },
  arrow/.style={-Stealth, thick}
]

\node[bubble] (S2) at (0, 2.5) {§2: Thm 1.2 \\ $\mathrm{IM}(2,q)$};
\node[bubble] (S3) at (4, 1) {§3: Thm 1.6 \\ $\mathrm{IM}(d,q)$};
\node[bubble] (S4) at (-4, 1) {§4:  Prop 4.1 \\ $\mathrm{IM}(2,p^t)$};
\node[bubble] (S6) at (0, -0.5) {§4-6: Thm 1.7--9 \\ $\mathrm{IM}(d,p^t)$};

\begin{pgfonlayer}{background}
  \node[cluster, fit=(S2)(S3)(S4)(S6)] (IM) {};
\end{pgfonlayer}
\node[font=\small, anchor=north] at ([yshift=-10pt]IM.north) {The induced matching problem};


\node[bubble] (S71)  at (-6,-4.5) {§7.1: Thm 1.10 \\ $\Nik(3+, q)$};
\node[bubble] (S72)  at (0.0,-4.5) {§7.2: Thm 1.11 \\ $\Nik(2, q)$};

\node[bubble] (S73)  at (0,-6.5) {§7.3: Thm 1.13 \\ $B(2, q)$};

\node[bubble] (S8) at (6,-4.5) {§8: Thm 1.15--17 \\ $\PL_{d}(\gamma)$};

\node[bubble] (S9) at (0,-8.5) {§9: Point--hyperplane matchings};

\draw[arrow] (S2) -- (S3);
\draw[arrow] (S2) -- (S4);
\draw[arrow] (S4) -- (S6);
\draw[arrow] (S3) -- (S6);

\draw[arrow] (IM) -- (S71);
\draw[arrow] (IM) -- (S72);
\draw[arrow] (IM) -- (S8);


\draw[arrow] (S72) -- (S73);

\end{tikzpicture}
\endgroup
\end{figure}

\section{Lifting Paley graphs and Ruzsa sets}\label{sec:prime-2d}
In this section, we discuss the ``baseline" lower bound $\IM(2 ,q)\gg q \log q$. Now there are a few different constructions of induced matchings of size $q\log q$ in $\II_q^{(2)}$, but we would like to show an argument inspired by an old paper of Sz\H{o}nyi \cite{Szonyi1992}. Then we shall use a new (but related) argument to prove our first Theorem \ref{thm:d=2}. This new argument will serve as the foundation of almost all our constructions.

First, recall that if $q$ is a prime power with characteristic $q_0\equiv 1\pmod 4$, the \emph{Paley graph} $G_q$ is the (undirected) graph on vertex set $\F_q$ in which
distinct $x,y$ are adjacent iff $x-y$ is a nonzero square. It is well known that $G_q$ is self-complementary (multiplication by a fixed nonsquare induces an isomorphism
to the complement), and hence the independence number $\alpha(G_q)$ and the clique number $\omega(G_q)$ satisfy the relation $\alpha(G_q)=\omega(G_q)$. See for example \cite{sachs1962selfcomplementary} for more details.

In \cite{Szonyi1992}, Sz\H{o}nyi observed that it is possible to use large independent sets in $G_q$ in order to construct large {\it{minimal blocking sets}} in the finite projective plane $\PG(2,q)$. From a minimal blocking  set, it is then easy to construct an induced matching in the point-line incidence graph $\II_q^{(2)}$. 

\begin{proposition}\label{prop:paley-to-im}
If $q$ is a prime power with characteristic $q_0\equiv 1\pmod 4$, then
\[
  \IM(2,q)\ \ge\ q\,\alpha(G_q)\;=\;q\,\omega(G_q).
\]
\end{proposition}
\begin{proof}
Let $I\subset\F_q$ be an independent set in $G_q$, i.e.\ $(I-I)\cap (D_2^\times)=\emptyset$. We define the set
$$P=\left\{(x,y) \in \mathbb{F}_{q}^{2}:\ x+y^2 \in I\right\}.$$
Clearly, $|P|=q|I|$. For each such point $(x,y) \in P$, let $\ell_{x,y} \subset \mathbb{F}_{q}^{2}$ be the line
\[
  \ell_{x,y}:=(x,y)+\{\,\lambda(-2y,1):\lambda\in\F_q\,\} = (x-2\lambda y,\ y+\lambda).
\]
Note that $(x,y) \in \ell_{x,y}$. Crucially, note also that all the points of $\ell_{x,y}$ satisfy
\[
  (x-2\lambda y)+(y+\lambda)^2 = (x+y^2)+\lambda^2=t+\lambda^2,
\]
for some $t \in I$. However, then $t + \lambda^2$ can't be in $I$ unless $\lambda =0$. This means that the line $\ell_{x,y}$ intersects $P$ only at the point $(x,y)$. We conclude that $\{((x,y),\ell_{x,y})\}_{(x,y)\in P}$ is an induced matching in $\II_q^{(2)}$.
\end{proof}
\noindent {\bf Some history.}
Determining the clique number $\omega(G_q)$ of the Paley graph is a classical problem at the intersection of several topics in combinatorics and number theory. The Delsarte--Hoffman eigenvalue method yields the standard ``square-root barrier''
\[
\omega(G_q)\le \sqrt q.
\]
See, for instance, \cite{Yip2021} for all the relevant background. 
When $q=q_0^2$, this barrier is sharp: indeed $\mathbb F_{q_0}^\times\subset (\mathbb F_{q_{0}^2}^\times)^2$, so the subfield $\mathbb F_{q_0}\subset \mathbb F_{q_{0}^2}$ spans a clique of size $q_{0}=\sqrt q$, and hence $\omega(G_{q_{0}^2})=\sqrt q$.
For prime $q\equiv 1\pmod 4$, Hanson--Petridis proved the best known constant-factor improvement,
$\omega(G_q)\le \sqrt{q/2}+1$ \cite{HansonPetridis2021}. Thus, Proposition \ref{prop:paley-to-im} highlights the difficulty of asymptotically improving the upper bound $\IM(2, q) \ll q^{3/2}$ when $q$ is prime. We will discuss more about this problem in Section \ref{sec:conclusion}. 

On the lower-bound side, a standard application of Ramsey's theorem produces the lower bound $\omega(G_q)=\Omega(\log q)$ (see for example \cite{Cohen1988PaleyClique}), and
work of Graham--Ringrose implies that for infinitely many primes $q$ one has
$\omega(G_q)=\Omega(\log q\,\log\log\log q)$ \cite{GrahamRingrose1990LeastNonres}, improving to
$\Omega(\log q\,\log\log q)$ under the Generalized Riemann Hypothesis (see, e.g., \cite{Montgomery1971TopicsMultNT}).
Therefore, Proposition~\ref{prop:paley-to-im} produces the lower bound $\IM(2, q) \gg q \log q$ for all primes $q$, and improves it to $\IM(2, q) = \Omega(q \log q \log\log \log q)$ for infinitely many primes $q$.

We now discuss the polynomial improvement over Sz\H{o}nyi's construction. 

\medskip
\noindent {\it{Proof of Theorem \ref{thm:d=2}}}. The idea is that instead of lifting an independent set in the Paley graph, it is more efficient to lift a so-called {\it{square--difference--free}} set over $\Z$. A set $A \subset \Z$ is square--difference--free if for any distinct $a, a'\in A$, the difference $a - a'$ is not a square. The problem of finding the largest size of a square--difference--free of $[N]$ is known as the Furstenberg--S\'{a}rk\H{o}zy problem \cite{Furstenberg, Sarkozy}. While there have been recent exciting developments regarding the upper bound \cite{GreenSawhney}, here we will only need an older non--trivial lower bound construction due to Ruzsa \cite{Ruzsa}, which was subsequently improved slightly by Beigel--Gasarch and Lewko \cite{BeigelGasarch, Lewko}.
\begin{lemma}
    \label{thm:Ruzsa-square}
    There exists a square--difference--free subset of $[N]$ with size at least $N^{0.7334}$.
\end{lemma}
We will use such a set to construct a large point-line induced matching in $\II_q^{(2)}$.
\begin{proposition}
\label{lem:Ruzsa-lift}
Let $q$ be a prime, and $A$ be a square--difference--free subset of $[\lfloor q / 10\rfloor]$. Then we have
$$\IM(2,q) = \Omega(q^{1/2} |A|).$$
\end{proposition}
One can regard Proposition \ref{lem:Ruzsa-lift} as a direct analogue of Proposition \ref{prop:paley-to-im}, where we lift a square--difference--free subset of $[q/10]$, instead of an independent set $I \subset \F_q$ in the Paley graph, to an induced matching in the point-line incidence graph $\II_q^{(2)}$. 
\begin{proof}
    Let $N = \lfloor q / 3 \rfloor$ and $M = \lfloor \sqrt{q} / 2 \rfloor$. We define the set 
    $$P = \{(x, y) \in [N] \times [M]: 2x - y^2 \in A\}$$
    and for each $p = (x,y) \in P$, define the associated line $\ell_p = \{(x, y) + t(y, 1): t\in \Z\}$. 
    
    We claim that $\{(p,\ell_p):p\in P\}$ is an induced matching in the point--line incidence graph of $\F_q^2$. For the sake of contradiction, assume that distinct $p = (x, y) \in P$ and $p' = (x', y') \in P$ satisfies $p' \in \ell_p$ over $\F_q$. Then there exists $t\in\F_q$ such that
$$(x',y') \equiv (x+ty,\ y+t) \pmod{q}.$$
In particular, from the second coordinate equality we have $t=y'-y \in \F_q$, which after substitution in the first equality gives
$$x' \equiv x+y(y'-y)\pmod q.$$
Because $x,x'\in [M]$ and $y,y'\in [N]$, we have
$$|x+y(y'-y)-x'| \le |x-x'|+|y|\,|y'-y| \le (M-1) + (N-1)^2 < \frac{q}{3}+\frac{q}{4} < q.$$
    so we have the equality over $\Z$
    \begin{equation} \label{Zeq}
    x + y(y' - y) = x'.
    \end{equation}
    Next, let $2x - y^2 = a \in A$ and $2x' - y'^2 = a' \in A$ (recall that $(x,y),(x',y') \in P)$. Using \eqref{Zeq}, it follows that
    $$a - a' = 2(x - x') + y'^2 - y^2 = 2 y(y - y') + y'^2 - y^2 = (y - y')^2,$$
    which is a perfect square. Since $a, a' \in A$ and $A$ is square--difference--free, we have $a = a'$, so $y = y'$ and $x = x'$, contradiction. 
    
    We have proved $\ell_p\cap P=\{p\}$ for every $p\in P$. In particular, distinct points of $P$ give distinct lines,
and the pairs $\{(p,\ell_p)\}_{p\in P}$ form an induced matching. Hence $\IM(2,q)\ \ge\ |P|$. It remains to estimate the size of $A$. Recall that $A \subset [q / 10]$. For each $a \in A$ and $y \in [M]$ with the same parity, we have $(a + y^2) / 2 \in [N]$, so the point $((a + y^2) / 2, y)$ is an element of $P$. Therefore, we obtain
    $$\abs{P} \geq (N - 1) / 2 \cdot |A| = \Omega(q^{1/2} |A|)$$
    as desired.
\end{proof}
Theorem~\ref{thm:d=2} now follows directly from Lemma~\ref{thm:Ruzsa-square} and Proposition~\ref{lem:Ruzsa-lift}.

\medskip
\section{Lifting to higher dimensions}
\label{sec:prime-highdim}
Next, we show how Sz\H{o}nyi's Paley lifting can be generalized to higher dimensions, leading us to the proof of Theorem~\ref{thm:prime-high-dim}.

Fix an integer $d\ge 2$. Write
\[
  D_d := \{t^d : t\in \F_q\}
  \qquad\text{and}\qquad
  D_d^\times := D_d\setminus\{0\}.
\]
The baseline lower bound $\IM(d,q) \gg q^{d-1} \log q$ comes from the following more general construction.
\begin{proposition}\label{prop:dth-power-highd}
Assume $\mathrm{char}(\F_q)>d$.
Let $I\subset \F_q$ satisfy
\[
  (I-I)\cap D_d^\times=\emptyset.
\]
Then we have
\[
  \IM(d,q)\ \ge\ \frac{1}{(d-1)!}\,|I|\,(q-1)^{d-1}.
\]
\end{proposition}

\begin{proof}
Set
\[
  \Phi(x_1,\dots,x_d):=x_1+x_2^2+\cdots+x_d^d.
\]
Our goal is to construct a set of points $P\subset \F_q^d$ and, for each $p\in P$, a line $\ell_p$ such that $\ell_p\cap P=\{p\}$. The pairs $(p,\ell_p)$ would then form an induced matching in $\II_q^{(d)}$ of size $|P|$.

To do so, we will first build $(d-1)$-tuples $(x_2,\dots,x_d)\in(\F_q^\times)^{d-1}$ together with a direction vector
$v=(v_1,\dots,v_d)\in(\F_q^\times)^{d-1}\times\{1\}$
 (a vector $v$ which will be a function of $x_1,\ldots,x_{d}$) such that for all $x_1,\lambda \in \mathbb{F}_{q}$, we have
\begin{equation}\label{eq:Phi-shift}
  \Phi(x_1+\lambda v_1,\dots,x_d+\lambda v_d)=\Phi(x_1,\dots,x_d)+\lambda^d.
\end{equation}
Expanding by the binomial theorem, \eqref{eq:Phi-shift} is equivalent to requiring that all coefficients
of $\lambda^r$ for all $1\le r\le d-1$ vanish, while the $\lambda^d$ coefficient equals $1$ (the coefficient of $\lambda^d$ is $v_d^d=1$).

For $1\le r\le d-1$, the $\lambda^r$ coefficient equals
\[
  v_r^r\;+\;\sum_{i=r+1}^d \binom{i}{r}\,x_i^{\,i-r}v_i^r,
\]
so we want
\begin{equation}\label{eq:vr-rec}
  v_r^r = -\sum_{i=r+1}^d \binom{i}{r}\,x_i^{\,i-r} v_i^r.
\end{equation}

We choose the variables in decreasing order $r=d-1,d-2,\dots,2$, and handle $r=1$ at the end. When $r=d-1$, the sum in \eqref{eq:vr-rec} contains only the term $i=d$. This gives
$$v_{d-1}^{\,d-1} = -{d \choose d-1} x_d^{\,1}v_d^{\,d-1} = -d x_d,$$
since $v_{d}=1$. This provides the base case of the recursion.

For $2\le r\le d-2$, now let us assume that the variables $x_{r+2},\dots,x_d$ and $v_{r+1},\dots,v_d$ have already been chosen, and are nonzero. Then in \eqref{eq:vr-rec} the only occurrence of the new variable $x_{r+1}$ comes from the term with $i=r+1$. For the reader's convenience, we single out this term
\[
-\sum_{i=r+1}^d \binom{i}{r}\,x_i^{\,i-r} v_i^r
= -\binom{r+1}{r}\,x_{r+1}\,v_{r+1}^r \;-\; \sum_{i=r+2}^d \binom{i}{r}\,x_i^{\,i-r} v_i^r.
\]
This allows us to emphasize that the right-hand side of \eqref{eq:vr-rec} is an affine-linear function of $x_{r+1}$ of the form
\[
A\,x_{r+1}+B,\ \ \ \ \ \text{where}
\qquad
A:=-(r+1)v_{r+1}^r,\quad
B:=-\sum_{i=r+2}^d \binom{i}{r}\,x_i^{\,i-r} v_i^r.
\]
Note that $B$ is already determined by the previously chosen variables.
Moreover $A\neq 0$ because $v_{r+1}\neq 0$ and $\mathrm{char}(\F_q)>d\ge r+1$.
Therefore the map $x_{r+1}\mapsto A x_{r+1}+B$ is a bijection of $\F_q$, so as $x_{r+1}$ ranges over $\F_q$
the right-hand side of \eqref{eq:vr-rec} ranges over all of $\F_q$.

In particular, we may choose $x_{r+1}\in\F_q^{\times}$ so that the right-hand side lies in $D_{r}^{\times}$ (recall $D_{r}^{\times}=\left\{t^r:t\in\F_q^\times\right\}$),
and then pick $v_r\in\F_q^\times$ satisfying \eqref{eq:vr-rec}.
Note that
$$\abs{D_r^{\times}} = \frac{q-1}{\gcd(r,q-1)} \geq \frac{q - 1}{r},$$
giving at least $(q-1)/r$ valid choices of $x_{r+1}$ for which the right-hand side
of \eqref{eq:vr-rec} lies in $D_{r}^{\times}$. For each of these, we can then choose $v_r\in\F_q^\times$ satisfying \eqref{eq:vr-rec}. 

After choosing $v_2,\dots,v_d$, we choose $v_1\in\F_q$ so that the $\lambda^1$ coefficient vanishes;
this is always possible because the $r=1$ equation \eqref{eq:vr-rec} is linear in $v_1$.

We have thus obtained a set $S\subset(\F_q^\times)^{d-1}$ of admissible tuples $(x_2,\dots,x_d)$, and for each $(x_2,\dots,x_d)\in S$ we fixed one corresponding direction $v(x_2,\dots,x_d)$ satisfying \eqref{eq:Phi-shift}. Moreover, we have also shown along the way that
$$|S|\ \ge\ \prod_{r=1}^{d-1}\frac{q-1}{r} = \frac{(q-1)^{d-1}}{(d-1)!}.$$
We can define finally the point set
$$P = \left\{p=(x_1, x_2, \cdots, x_d) \in \mathbb{F}_{q}^{d}:\ (x_2, \cdots, x_t) \in S\ \ \text{and}\ \ \Phi(p) = x_1 + \sum_{i=2}^d x_i^i \in I\right\}.$$
By construction, we have
\[
  |P|=|I|\,|S|\ \ge\ \frac{1}{(d-1)!}\,|I|\,(q-1)^{d-1}.
\]
For $p = (x_1, x_2, \cdots, x_t) \in P$ arising from $(x_2,\dots,x_d)\in S$, define the affine line
\[
  \ell_p:=\{p+\lambda v(x_2,\dots,x_d):\ \lambda\in\F_q\}.
\]
By \eqref{eq:Phi-shift}, for every $\lambda\in\F_q$ we have
$\Phi(p+\lambda v)=\Phi(p)+\lambda^d$.
If $p'\in \ell_p\cap P$, then $\Phi(p'),\Phi(p)\in I$ and
\[
  \Phi(p')-\Phi(p)=\lambda^d\in D_d.
\]
By the hypothesis $(I-I)\cap D_d^\times=\emptyset$, this forces $\lambda=0$, hence $p'=p$.
Thus $\ell_p\cap P=\{p\}$ for every $p\in P$, and $\{(p,\ell_p)\}_{p\in P}$ is an induced matching of size $|P|$.
\end{proof}

\medskip

\begin{remark}
The condition from Proposition \ref{prop:dth-power-highd} simply means that $I$ is an independent set in the Cayley graph $G_{q,d}$ on the additive group of $\F_q$, generated by the nonzero $d$th powers.
Equivalently, this is the graph on vertex set $\F_q$ in which two distinct elements are adjacent if and only if their difference is a nonzero $d$th power. When $q$ is a prime power, this Cayley graph is highly symmetric and behaves in many respects like the usual Paley graph. 

When $q$ is a prime power, $G_{q,d}$ is a normal Cayley graph and its spectrum is controlled by classical
character-sum estimates; in particular, all nontrivial eigenvalues are $O_d(\sqrt q)$.
The Delsarte--Hoffman bound then gives
\[
  \alpha(G_{q,d})\ \ll_d\ \sqrt q,
\]
matching the familiar ``square-root barrier'' for Paley-type graphs.
Moreover, in many extension-field situations this is sharp up to constants; for example, when $q=q_0^2$ and
$D_d^\times$ contains $\F_{q_0}^\times$ (equivalently $\gcd(d,q-1)\mid(q_0+1)$), the subfield $\F_{q_0}\subset\F_q$
forms a clique of size $q_0$, and by scaling by an element outside $D_d^\times$ one obtains an independent set of
size $q_0$ as well (see, e.g., \cite{Yip2021}).

In the prime-field case, much less is known about the extremal behaviour of $G_{q,d}$.
First, if $\gcd(d,q-1)=1$ then $D_d^\times=\F_q^\times$, so the condition $(I-I)\cap D_d^\times=\emptyset$ forces $|I|=1$,
and Proposition~\ref{prop:dth-power-highd} yields only the trivial lower bound $\IM(d,q)\gg_d q^{d-1}$.
Otherwise $\gcd(d,q-1)>1$, so $D_d^\times$ is a proper multiplicative subgroup.
In the undirected situation ($-1\in D_d^\times$), one can convert cliques into independent sets by dilation:
if $C$ is a clique and $\xi\notin D_d^\times$, then $\xi C$ is an independent set (since
$\xi D_d^\times$ is a disjoint coset of $D_d^\times$).
Thus $\alpha(G_{q,d})\ge \omega(G_{q,d})$, and combining this with the standard Ramsey bound
$\max\{\alpha(G_{q,d}),\omega(G_{q,d})\}\gg \log q$ produces an independent set $I\subset\F_q$ of size $\gg\log q$.
Consequently Proposition~\ref{prop:dth-power-highd} gives the baseline estimate (for such primes~$q$)
\[
  \IM(d,q)\ \gg_d\ q^{d-1}\log q.
\]
Obtaining substantially larger sets $I$ in prime fields appears comparable in difficulty to the classical
problem of finding large independent sets in Paley graphs, and Proposition~\ref{prop:dth-power-highd} by itself does not lead to any substantial improvements for $\IM(d, q)$.

\end{remark}

On the other hand, Proposition~\ref{prop:dth-power-highd} naturally leads to the idea of a ``higher power analogue" of Lemma~\ref{lem:Ruzsa-lift}. We next establish such a result and prove Theorem~\ref{thm:prime-high-dim}.

We will mostly work with integers, reducing modulo $q$ only in the final step. All equations in the rest of this section, unless otherwise indicated, hold over $\Z$. 

We need several preliminary results.
\begin{lemma}[Waring's problem]
    \label{lem:waring}
    For each positive integer $k$, there exists a positive integer $G(k)$ and a threshold $T(k)$ such that for any non--negative integer $n \geq T(k)$, the equation in $G(k)$ variables
    $$x_1^k + \cdots + x_{G(k)}^k = n$$
    has a solution in non--negative integers.
\end{lemma}
Classical results of Vinogradov \cite{VinogradovWaring} show that we can take $G(k) = O(k \log k)$, which has remained the best known upper bound to date, up to lower order terms. See for example the survey of Vaughan and Wooley \cite{Waring} for further progress.

The next lemma is the higher power analogue of Lemma~\ref{thm:Ruzsa-square}. It is established alongside Lemma~\ref{thm:Ruzsa-square} (with exponent $0.7330\cdots$) in Ruzsa's paper \cite{Ruzsa}, and we reproduce its short proof here for the reader's convenience.
\begin{lemma}
    \label{lem:$k$--th--power--difference--free}
    Let $N$ be a positive integer, and let $D_k(N)$ be the size of the largest $A \subset [N]$ such that $A - A$ avoids non--zero $k$--th powers. If $p$ is the smallest prime congruent to $1$ modulo $2k$, then we have
    $$D_k(N) = \Omega_k\left(N^{c_k}\right)$$
    where
    $$c_k = 1 - \frac{1}{k} + \frac{\log k}{k \log p}.$$
    In particular, if $2k + 1$ is a prime, then we have
    $$c_k = 1 - \frac{1}{k} + \frac{\log k}{k \log (2k + 1)} \geq 1 - \frac{2}{k \log k}.$$
\end{lemma}
\begin{proof}
    Let $Q$ be the set of nonzero $k$-th powers in  $\Z / p\Z$. As $p \equiv 1 \bmod{2k}$, we have $Q = -Q$ and $|Q| = \frac{p - 1}{k}$. By the greedy algorithm, we can sample $A_p = \{a_1, \cdots, a_k\}$ in $\Z / p\Z$ such that $a_{i + 1}$ does not lie in $(a_1 + Q) \cup \cdots \cup (a_i + Q)$. Hence, we have $|A_p| \geq k$ and $A_p - A_p$ contains no nonzero $k$-th powers in  $\Z / p\Z$.
    
    Let $L = \lfloor{\log_{p} N \rfloor}$. Let $A\subset [N]$ be the set of positive integers $$n=1 + \sum_{i=0}^{L - 1} c_i p^i$$
    where $c_i \in \{0, 1, \cdots, p - 1\}$, and $c_d\in A_p$ whenever $k\mid d$. It is simple to check that $A - A$ contains no nonzero $k$--th power. Setting $s = \lceil L / k \rceil$, we have
    $$|A| \geq p^{L - s} |A_p|^s \gg_k N \cdot \left(\frac{k}{p}\right)^{\log_p N / k} = N^{c_k}$$
    as desired.
\end{proof}
We now begin the proof of Theorem~\ref{thm:prime-high-dim}. Fix $k$ to be a positive integer such that $2k + 1$ is a prime. Set $\ell = k^2$. We define an index set
$$I := I_k = \{(\alpha, \beta, \gamma): \alpha \in [k - 1], \beta \in [0, \ell (k - \alpha)], \gamma \in [2G(\alpha)]\} \cup \{0\}.$$
Here $G(\cdot)$ is the number of variables in Lemma~\ref{lem:waring}, and we use the notation
$$[n] := \{1,2,\cdots, n\}, \quad [0, n] := \{0,1,2,\cdots, n\}.$$

Let $N$ be a positive integer and set $M = \lfloor N^{1 / \ell} \rfloor$.  We define a polynomial $\Phi_N \in \Z[x_{I}]$ by
$$\Phi_N(x_I) = x_0^k + \sum_{\alpha \in [k - 1]}\sum_{\beta \in [0, \ell (k - \alpha)]} M^{\beta} \left(\sum_{\gamma = 1}^{G(\alpha)} x_{\alpha, \beta, \gamma}^{\alpha} -  \sum_{\gamma = G(\alpha) + 1}^{2G(\alpha)} x_{\alpha, \beta, \gamma}^{\alpha}\right).$$
This polynomial satisfies the following lemma, which adapts \eqref{eq:Phi-shift} to the integer setting.
\begin{lemma}
    \label{lem:nice--line}
    For any $x_I \in [N]^I$, there exists some $y_I \in \Z^I$ such that, as a polynomial identity in $h$, we have
    $$\Phi_N(x_I + y_I h) = \Phi_N(x_I) + h^k,$$
    and furthermore, we have $\infnorm{y_I} = O_k(M)$.
\end{lemma}
\begin{proof}
    For each $\alpha \in \{1, \cdots, k - 1\}$. Let $I_{\alpha} = \{(\alpha, \beta, \gamma): (\alpha, \beta, \gamma) \in I\}$ and $y_{\alpha} = y_{I_{\alpha}}$. Note that the term in $\Phi_N(x_I + y_I h)$ with index $(\alpha, \beta, \gamma)$ only contribute to the coefficients of $h^{\alpha'}$ with $\alpha' \leq \alpha$. 
    
    We set $y_0 = 1$, and iteratively choose $y_{\alpha}$ in decreasing order of $\alpha = k - 1 , \cdots, 1$, such that the coefficient of $h^{\alpha}$ in $\Phi_N(x_I + y_I h) $ is zero. This coefficient is given by
    $$\binom{k}{\alpha} x_{0}^{k - \alpha} y_0^{\alpha} + \sum_{\epsilon = \alpha}^{k - 1} \sum_{\beta \in [0, \ell (k - \alpha)]} M^{\beta} \left(\sum_{\gamma = 1}^{G(\alpha)} \binom{\epsilon}{\alpha} x_{\epsilon, \beta, \gamma}^{\epsilon - \alpha} y_{\epsilon, \beta, \gamma}^{\alpha} -  \sum_{\gamma = G(\alpha) + 1}^{2G(\alpha)} \binom{\epsilon}{\alpha} x_{\epsilon, \beta, \gamma}^{\epsilon - \alpha} y_{\epsilon, \beta, \gamma}^{\alpha}\right).$$
    Isolating the term corresponding to $I_{\alpha}$, we need to ensure that
    $$\sum_{\beta \in [0, \ell (k - \alpha)]} M^{\beta} \left(\sum_{\gamma = 1}^{G(\alpha)} y_{\alpha, \beta, \gamma}^{\alpha} -  \sum_{\gamma = G(\alpha) + 1}^{2G(\alpha)} y_{\epsilon, \beta, \gamma}^{\alpha}\right) = R$$
    where
    $$R = - \binom{k}{\alpha} x_{0}^{k - \alpha} y_0^{\alpha} - \sum_{\epsilon = \alpha + 1}^{k - 1} \sum_{\beta \in [0, \ell (k - \alpha)]} M^{\beta} \left(\sum_{\gamma = 1}^{G(\alpha)} \binom{\epsilon}{\alpha} x_{\epsilon, \beta, \gamma}^{\epsilon - \alpha} y_{\epsilon, \beta, \gamma}^{\alpha} -  \sum_{\gamma = G(\alpha) + 1}^{2G(\alpha)} \binom{\epsilon}{\alpha} x_{\epsilon, \beta, \gamma}^{\epsilon - \alpha} y_{\epsilon, \beta, \gamma}^{\alpha}\right).$$
    Estimating the size of $R$ using the triangle inequality, we have
    $$\abs{R} \ll_k N^{k - \alpha} + \sum_{\epsilon > \alpha}^{k - 1} M^{\ell (k - \alpha)} \infnorm{y_{\epsilon}}^{\alpha} \ll_k  N^{k - \alpha} (1 + \max_{\epsilon > \alpha} \infnorm{y_{\epsilon}}^{\alpha}).$$
    Expressing $R$ in base $M$, we can write
    $$R = \sum_{\beta \in [0, \ell(k - \alpha)]} R_{\beta} M^{\beta}$$
    where we have 
    $$\abs{R_{\beta}} \leq \max\left(M, \frac{\abs{R}}{M^{\ell (k - \alpha)}}\right) \ll_k \max\left(M, \max_{\epsilon > \alpha} (1 + \infnorm{y_{\epsilon}}^{\alpha}) \right).$$
    Now Lemma~\ref{lem:waring} shows that there exists $\{y_{\alpha, \beta, \gamma}\}_{\gamma \in [2G(\alpha)]}$ that satisfy
    $$\sum_{\gamma = 1}^{G(\alpha)} y_{\alpha, \beta, \gamma}^{\alpha} = \max(R_{\beta}, 0) + T(\alpha)$$
    and 
    $$\sum_{\gamma = G(\alpha) + 1 }^{2G(\alpha)} y_{\alpha, \beta, \gamma}^{\alpha} = \max(-R_{\beta}, 0) + T(\alpha).$$
    So we have
    $$\sum_{\gamma = 1}^{G(\alpha)} y_{\alpha, \beta, \gamma}^{\alpha} -  \sum_{\gamma = G(\alpha) + 1}^{2G(\alpha)} y_{\epsilon, \beta, \gamma}^{\alpha} = R_{\beta}$$
    and therefore
    $$\sum_{\beta \in [0, \ell (k - \alpha)]} M^{\beta} \left(\sum_{\gamma = 1}^{G(\alpha)} y_{\alpha, \beta, \gamma}^{\alpha} -  \sum_{\gamma = G(\alpha) + 1}^{2G(\alpha)} y_{\epsilon, \beta, \gamma}^{\alpha}\right) = \sum_{\beta \in [0, \ell (k - \alpha)]} M^{\beta} R_{\beta} = R$$
    as desired. Furthermore, we have $\abs{y_{\alpha, \beta, \gamma}}^{\alpha} \leq (\abs{R_{\beta}} + T(\alpha))^{1 / \alpha}$, so we obtain
    $$\infnorm{y_{\alpha}} \leq \max_{\beta} (\abs{R_{\beta}} + T(\alpha))^{1 / \alpha} \ll_k \max\left(M^{1 / \alpha}, \max_{\epsilon > \alpha}(1 + \infnorm{y_{\epsilon}})\right).$$
    Telescoping, we obtain $\infnorm{y_{\alpha}} \ll_k M^{1 / \alpha}$ for each $\alpha \in [k - 1]$, and in particular $\infnorm{y_{I}} \ll_k M$, as desired.
\end{proof}
We are now ready to construct the induced point--line matching. Let $q$ be a large prime, and set $N = \lfloor{q / 4\rfloor}$. Let $A \subset [q^k]$ be the largest subset of $[q^k]$ such that $A - A$ avoids non--zero $k$--th powers. As before, set $I = I_k$. For a shift $s \in \mathbb{Z}$, we define the set
$$\Gamma_s = \{x_I \in [N]^{I}: \Phi_N(x_I) \in A + s\}.$$
We observe that for any $x_I \in [N]^{I}$, we have
$$\abs{\Phi_N(x)} \ll_k x_0^{k} + \max_{\alpha \in [k - 1]} M^{\ell (k - \alpha)} N^{\alpha} \ll_k N^k.$$
Hence, $\abs{\Gamma_s} = 0$ unless $s \ll_k q^k$ (recall $N \asymp q$), and we have
$$\sum_{s \in \Z} \abs{\Gamma_s} = |A| N^{|I|}.$$
Therefore, we can choose some $s \in \Z$ such that
$$\abs{\Gamma_s} \gg_k \frac{|A| N^{|I|}}{q^k}.$$
By Lemma~\ref{lem:nice--line}, for any $x_I \in [N]^I$, there exists some $y_I \in \Z^I$ such that 
$$\Phi_N(x_I + y_I h) = \Phi_N(x_I) + h^k$$
and $\infnorm{y_I} \leq C_k M$, where $C_k > 0$ is some constant depending only on $k$.
We now consider the subset of $[q] \times [q]^{I}$ defined by
$$P := \{(z, x_I) \in [q] \times [q]^I: z \in [N / C_k M], x_I \in \Gamma_s\}.$$
For each $p = (z, x_I)$ in $P$, define the line $\ell_p = \{(z + h, x_I + y_I h): h \in \Z\}$. We claim that the reduction of $P$ and $\{\ell_p\}_{p \in P}$ modulo $q$ is an induced point--line matching.

Suppose for the sake of contradiction that for some distinct points $p = (z, x_I)$ and $p' = (z', x_I')$ in $P$, we have $p' \in \ell_{p}$ when reduced modulo $q$. Let $h = z' - z$. Then we have $\abs{h} \leq N / C_k M$. For each coordinate $i \in I$, we have
$$x'_i \equiv x_i + hy_i \bmod{q}.$$
Furthermore, we have
$$\abs{x_i' - x_i -hy_i} \leq N + N + \frac{N}{C_k M} \cdot C_k M  = 3N < q.$$
So we must have $x_i' = x_i + hy_i$ over $\Z$. Therefore, we have $x_I' = x_I + hy_I$ over $\Z$. By definition, we obtain (over $\Z$)
$$\Phi_N(x_I') = \Phi_N(x_I) + h^k.$$
However, both $\Phi_N(x_I)$ and $\Phi_N(x_I')$ are elements of $A + s$, and $A - A$ avoids non--zero $k$--th powers. Therefore, we must have $h = 0$ and $p = p'$, contradiction.

We conclude that the reduction of $P$ and $\{\ell_p\}$ modulo $q$ is an induced point--line matching in $\F_q^{|I| + 1}$. Its size is given by
$$\abs{P} \gg_k \frac{N}{M} \cdot\frac{|A| N^{|I|}}{q^k} \gg_k q^{|I| + 1 - k(1 - c_k) - \ell^{-1}} = q^{|I| + 1 - k(1 - c_k) - k^{-2}}$$
where $c_k$ is the constant in Lemma~\ref{lem:$k$--th--power--difference--free}. Substituting the value of $c_k$ in that lemma, we conclude the following.
\begin{proposition}
    \label{prop:Ikd}
    Let $k$ be a positive integer such that $2k + 1$ is a prime, and let $d = \abs{I_k} + 1$. For prime field $\F_q$, we have
    $$\IM(d, q) \gg_k q^{d - k(1 - c_k) - k^{-2}} \geq q^{d - \frac{2}{\log k} - \frac{1}{k^2}}.$$
\end{proposition}
 We have $\IM(d + 1, q) \geq q \IM(d, q)$ by \eqref{eq:IMlift}.
 Hence, Proposition~\ref{prop:Ikd} also holds when $d \geq \abs{I_k} + 1$. Thus for each dimension $d$, we have
$$\IM(d, q) \gg_d q^{d - \epsilon_d},\ \ \text{with}\ \epsilon_d := \frac{2}{\log k} + \frac{1}{k^2},$$
where $k$ is the greatest integer such that $2k + 1$ is prime and $d \geq \abs{I_k} + 1$.

By the construction of $I_k$ and Vinogradov's estimate for the Waring number $G(k)$, we have $\abs{I_k} = k^{O(1)}$. Bertrand's postulate shows that $k = d^{\Omega(1)}$ and thus $\epsilon_d \ll (\log d)^{-1}$. This completes the proof of Theorem~\ref{thm:prime-high-dim}.

\section{Lifting for prime powers: a warm-up}

We now shift the attention to base fields $\F_q$ where $q$ is a prime power. The goal of this section is to prove the subsequent Proposition~\ref{prop:ruzsa-lift-prime-power}, as well as Corollary~\ref{cor:prime-power}. Both of these results admit relatively short proofs, and serve as motivations for the remaining arguments. 

In previous sections, we used the isomorphism $\F_q \cong \Z / q\Z$ to reduce various constructions over $\Z$ to the desired constructions over $\F_q$. While this isomorphism no longer holds when $q = p^t$ is a prime power, we instead have
$$\F_q \cong \Z[X] / (p, \phi(X))$$
where $\phi$ is an irreducible polynomial in $\F_p[X]$ of degree $t$. One natural strategy is to do constructions over $\Z[X]$, and reduce them via this isomorphism to constructions over $\F_q$. This strategy leads to the following Proposition~\ref{prop:ruzsa-lift-prime-power}.
\begin{proposition}
\label{prop:ruzsa-lift-prime-power}
Let $t\ge 1$ be an \emph{odd} integer, let $p$ be an odd prime such that $p >100t$, and set $q=p^t$. Let $A$ be a square--difference--free subset of $[\lfloor p/(20t)\rfloor]$.
Then, 
\[
  \IM(2,q) \gg_t p^{\frac{3t-1}{4}}\,|A|^{\frac{t+1}{2}} = q^{\frac{3t-1}{4t}}\;|A|^{\frac{t+1}{2}}.
\]
In light of Lemma~\ref{thm:Ruzsa-square}, we thus have
$$\IM(2,q) \gg_t q^{\frac{3t-1}{4t}}\; p^{\frac{t+1}{2} \cdot 0.7334} \gg_t q^{1.1167}.$$
\end{proposition}
For $t=1$, note that this recovers Proposition~\ref{lem:Ruzsa-lift}. It is perhaps important to add that a similar “degree-barrier” construction can be carried out for even $t$ as well (taking $\deg f < t/2$ so that $\deg(f^2)<t$).
However, when $t$ is even the field size $q=p^t$ is a square, and the Hermitian unital construction already yields induced matchings of size $\asymp q^{3/2}$, which is substantially larger than what the Ruzsa-type lifting provides in this regime. For this reason we focus on odd $t$, where no unital-type construction is available.

\begin{proof}
Choose $\alpha\in \F_q$ such that $\F_q=\F_p(\alpha)$ and let $\phi\in\F_p[X]$ be the minimal monic polynomial of $\alpha$.
Then $\deg\phi=t$, and every element of $\F_q$ has a unique representative polynomial of degree $<t$ in $\F_p[X]$
via the identification $\F_q\cong \F_p[X]/(\phi)$.

Let $s=\frac{t+1}{2} \leq t$, $M = \left\lfloor \sqrt{\frac{p}{64s}}\right\rfloor$, and define $X(s,M)$ to be the set of integer polynomials
 \[
 f(X)=\sum_{i=0}^{s-1} f_i X^i \in \Z[X]
  \qquad\text{with}\qquad |f_i|\le M.
  \]
  Furthermore, let $Y_{even}(t,A)$ be the set of integer polynomials
  \[
  g(X)=\sum_{i=0}^{t-1} g_i X^i \in \Z[X]
  \]
  with $g_{2i} \in A$ for $i=0,1,\ldots,s-1$ and $g_{2i+1} \in \left\{0,1,\ldots,p-1\right\}$ for $i=0,1,\ldots,s-2$. In other words, the even coefficients of $g$ come from $A$, while the odd coefficients of $g$ are arbitrary elements of $\F_p$. For each pair $(f,g)\in X(s,M)\times Y_{even}(t,A)$, we can now define
  \[
  p_{f,g}:=\left( \frac{g(\alpha)+f(\alpha)^2}{2},f(\alpha) \right) \in \F_{q}^{2}.
  \]
  Note that $p_{f,g} \in \F_q^2$ indeed holds by design (and also because $p$ is odd). We also would like to emphasize that if $p_{f,g}=(x,y) \in \F_{q}^{2}$ then $f(\alpha)=y$ and $g(\alpha)=2x-y^2$. Finally, let 
  \[
  P := \left\{p_{f,g}:\ f \in X(s,M), g \in Y_{even}(t,A)\right\} \subset \F_{q}^{2}.
  \]
  This will be the point set defining our large induced matching in $\II_{q}^{(2)}$. Before, we define the set of lines, we take a moment to note that the map $X(s,M)\times Y_{even}(t,A) \to \F_{q}^2$ defined by $(f,g) \mapsto p_{f,q}$ is injective. Indeed, this is because $\deg f < s \leq t$, $\deg g < t$, and $M < p/2$, so the values of $f(\alpha)$ and $g(\alpha)$ uniquely recover $f$ and $g$. Hence
\begin{equation*} \label{sizeofP}
  |P|=|X(s,M)|\cdot |Y_{\mathrm{even}}(t,A)|
  =(2M+1)^s\cdot (p^{s-1}|A|^s) \gg_{t} p^{\frac{3t-1}{4}}\,|A|^{\frac{t+1}{2}}
   = q^{\frac{3t-1}{4t}}\;|A|^{\frac{t+1}{2}}.
\end{equation*}
For $p=(x,y)\in P$, let us now define the affine line
\[
  \ell_p:=\{(x,y)+\lambda(y,1):\ \lambda\in\F_q\}
  =\{(x+\lambda y,\ y+\lambda):\lambda\in\F_q\}.
\]
We claim that $\{(p,\ell_p):p\in P\}$ is an induced matching in the point--line incidence graph of $\F_q^2$. By the calculation from \eqref{sizeofP}, this will complete the proof of Proposition~\ref{prop:ruzsa-lift-prime-power}.

Like before, the key point is that for every $\lambda\in\F_q$ we have the identity
\begin{equation}\label{eq:square-shift-even}
  2(x+\lambda y)-(y+\lambda)^2 = (2x-y^2)-\lambda^2,
\end{equation}
so moving along $\ell_p$ changes $2x-y^2$ by a square. 

Fix $p=p_{f,g}\in P$ and suppose $p'=p_{f',g'}\in P$ lies on $\ell_p$.
Then $p'=(x+\lambda y,\,y+\lambda)$ for some $\lambda\in\F_q$, and by \eqref{eq:square-shift-even} we get
\[
  g(\alpha)-g'(\alpha) = (2x-y^2)-(2x'-y'^2)=\lambda^2.
\]
Also $y'=y+\lambda$ implies
\[
  \lambda = y'-y = f'(\alpha)-f(\alpha)=h(\alpha),
  \qquad\text{where }h:=f'-f.
\]
Thus
\begin{equation}\label{eq:eval-identity-even}
  (g-g')(\alpha)=h(\alpha)^2.
\end{equation}

Because $\deg h < s$ and $t=2s-1$, we have $\deg(h^2)\le 2(s-1)<t$.
Since evaluation at $\alpha$ is injective on degree-$<t$ polynomials over $\F_p$,
\eqref{eq:eval-identity-even} implies the polynomial identity in $\F_p[X]$:
\begin{equation}\label{eq:poly-identity-even}
  g(X)-g'(X)\equiv h(X)^2 \pmod p.
\end{equation}

Now let $i\ge 0$ be the smallest index with the $X^i$ coefficient of $h$ nonzero in $\F_p$.
Write that coefficient as $d_i\in\{-2M,\dots,2M\}\subset\Z$ (using $2M<p/2$ so representatives are unambiguous).
Then the coefficient of $X^{2i}$ in $h^2$ is $d_i^2$, and it is the first nonzero coefficient of $h^2$.

Consequently, the first nonzero coefficient of $g-g'$ must also occur in degree $2i$.
In particular, $g_{2j}=g'_{2j}$ for all $j<i$, and at degree $2i$ we have
\[
  g_{2i}-g'_{2i}\equiv d_i^2 \pmod p.
\]
But $g_{2i},g'_{2i}\in A\subset[\,p/(20t)\,]$, so $g_{2i}-g'_{2i}\in A-A\subset(-p/2,p/2)$ as an integer.
Also $|d_i|\le 2M$ and $4M^2\le p/(16s)<p/2$, hence $d_i^2\in(-p/2,p/2)$.
Therefore the congruence mod $p$ lifts to the integer equality
\[
  g_{2i}-g'_{2i}=d_i^2\in\Z.
\]
If $g\neq g'$, then $g_{2i}-g'_{2i}\neq 0$ and we have produced a nonzero perfect square in $A-A$,
contradicting the square--difference--free hypothesis on $A$.
Hence $g=g'$. With $g=g'$, \eqref{eq:poly-identity-even} gives $h^2\equiv 0$, so $h\equiv 0$ and thus $f=f'$.
Therefore $p'=p$.

We have shown $\ell_p\cap P=\{p\}$ for every $p\in P$, so the pairs $\{(p,\ell_p)\}_{p\in P}$
form an induced matching of size $|P|$ in $\II_q^{(2)}$.
\end{proof}
We now turn our attention to Corollary~\ref{cor:prime-power}. We prove it using the isomorphism $\F_q \cong \F_p^t$ as $\F_p$--vector spaces.
\begin{proposition}
    \label{prop:prime-to-prime-power}
    Suppose $q=p^s$. Then for $d_0\ge 1$, we have $\IM(d_0 s,q)\ge \IM(d_0,p)^{s} \cdot q^{(s - 1)d_0}$.
\end{proposition}
\begin{proof}
    Recall that $\F_q\cong \F_p[T]/(f)$ for some irreducible $f$ of degree $s$.
    Each element $a$ of $\F_q$ can be uniquely written as $a = \sum_{i=0}^{s-1} c_iT^i +(f)$ with $c_i\in \F_p$. Let $\pi_i: \F_q \to \F_p$ map $a \in \F_q$ to $c_i \in \F_p$. Note that $\pi_i$ is a $\F_p$--linear map.

    Now let $M$ be an induced matching in $\F_p^{d_0}$ with point-set $P$, and each $p\in P$ having some direction $v_p\in \F_p^{d_0}\setminus\{0\}$ which spans a line $\ell_p$ with $P\cap \ell_p = \{p\}$. We then define an induced matching $M$ in $\F_q^{sd_0}$ as follows. For each $a = (a_j)_{j\in [sd_0]} \in \F_q^{sd_0}$ and each $i \in \{0, 1, \cdots, s - 1\}$, define $\phi_i(a) \in \F_p^{d_0}$ as
    $$\phi_i(a) := (\pi_i(a_j))_{j\in \{id_0+1,\dots,(i+1)d_0\}}.$$
    Our point-set shall be 
    $$P' = \left\{a \in \F_q^{sd_0} : \phi_i(a)\in P\text{ for all }i\in \{0,\dots,s-1\}\right\}.$$ Given a point $a \in P'$, let its corresponding line $\ell_a'$ have slope equal to
    $$v_a' = (v_{\phi_1(a)}, v_{\phi_2(a)}, \cdots, v_{\phi_s(a)}) \in \F_p^{sd_0}.$$
    Now suppose for the sake of contradiction that, for some $a \in P'$ and $\lambda \in \F_q \backslash \{0\}$, we have $a + \lambda v_a' \in P$. For each $i \in \{0, 1, \cdots, s - 1\}$ and $j\in \{id_0+1,\dots,(i+1)d_0\}$, we have
    $$\pi_i((a + \lambda v_a')_j) = \pi_i(a_j) + \pi_i(\lambda (v_{a}')_j).$$
    Since $(v_{a}')_j \in \F_p$, we have
    $$\pi_i((a + \lambda v_a)_j) = \pi_i(a) + (v_{a}')_j \pi_i(\lambda).$$
    Hence, we obtain
    $$\phi_i(a + \lambda v_a) = \phi_i(a) + \pi_i(\lambda) v_{\phi_i(a)}.$$
    Therefore, we have $\phi_i(a) + \pi_i(\lambda) v_{\phi_i(a)} \in P$. Since $P \cap \ell_{\phi_i(a)} = \{\phi_i(a)\}$, we obtain $\pi_i(\lambda) = 0$ for each $i \in \{0, 1, \cdots, s - 1\}$, so $\lambda = 0$, contradiction.

    We conclude that $(a, \ell_a')_{a \in P'}$ forms an induced point--line matching in $\F_q^{d_0 s}$. Its size is given by
    $$\abs{P'} = (|P| \cdot p^{d_0(s - 1)})^s = |P|^s \cdot q^{d_0(s - 1)}$$
    as desired.
\end{proof}
\begin{proof}[Proof of  Corollary~\ref{cor:prime-power}]
    If $s \geq \log d$, we may take $\epsilon_{d, s} = 1$ and apply the trivial bound $\IM(d, q) \geq q^{d - 1}$. So we may assume that $s < \log d$. Let $d_0 = \lfloor d / s \rfloor$. By \eqref{eq:IMlift}, we have
    $$\IM(d, q) \geq \IM(d_0 s, q) \cdot q^{d - d_0 s}.$$
    By Proposition~\ref{prop:prime-to-prime-power} and Theorem~\ref{thm:prime-high-dim}, we have
    $$\IM(d_0 s, q) \geq \IM(d_0,p)^{s} \cdot q^{(s - 1)d_0} \gg_d p^{sd_0 s- s \epsilon_{d_0}} q^{(s - 1)d_0}.$$
    Combining the estimates, we conclude that
    $$\IM(d, q) \gg_d q^{d - \epsilon_{d_0} s}.$$
    As $d_0 = \lfloor d / s \rfloor \gg d^{1/2}$, Corollary~\ref{cor:prime-power} holds with
    $$\epsilon_{d, s} := \epsilon_{d_0} s \ll \frac{s}{\log d}$$
    as desired.
\end{proof}

\section{Norm hypersurfaces over finite fields}
\label{sec:matching-highpower-II}

In this section we prove Theorem~\ref{thm:small-power} by constructing large induced matchings over $\F_q^d$ when $q = q_0^k$ is a prime power with $k \leq d$. We first describe the proof strategy. Recall that in Section~\ref{sec:prime-highdim}, the key ingredient is a polynomial $\Phi: \F_q^d \to \F_q$ such that for certain $x \in \F_q^d$, there exists some $y \in \F_q^d$ such that
$$\Phi(x + y h) = \Phi(x) + h^d.$$
Here, we instead consider the \emph{norm trace polynomial} $\Phi': \F_q^d \to \F_{q_0}$ given by
$$\Phi'(x) = \sum_{i = 1}^k \Norm(x_i).$$
Related norm--trace varieties appear in finite geometry and coding theory, e.g. \cite{Geil2003NormTrace}. We prove that for many $x \in \F_q^d$, there exists some $y \in \F_{q_0}^d$ such that the following relations hold for every $h \in \F_q$
$$\Phi'(x + y h) = \Phi'(x) + \Norm(h).$$
Hence, we can adapt the proof of Proposition~\ref{prop:dth-power-highd} in a much more efficient way.

\smallskip

Assume that $q=q_0^k$ with $k\ge 2$ and $q_0$ a prime (in fact, our proof actually also works when $q_0$ is a prime power).
Let $\sigma:\F_q\to\F_q$ be the Frobenius automorphism $\sigma(x)=x^{q_0}$.
Write
\[
  \Norm(x):=\prod_{j=0}^{k-1}\sigma^j(x)\in\F_{q_0}
\]
for the relative norm map $\F_q\to\F_{q_0}$.
For $\mu\in\F_q$ and $0\le r\le k$ set
\[
  e_r(\mu):=\sum_{0\le i_1<\cdots<i_r\le k-1}\sigma^{i_1}(\mu)\cdots\sigma^{i_r}(\mu)\in \F_{q_0},
\]
with the conventions $e_0(\mu)=1$ and $e_k(\mu)=\Norm(\mu)$.

\begin{lemma}\label{lem:norm-expansion-k}
For every $\mu\in\F_q$ and every $t\in\F_{q_0}$ one has
\[
  \Norm(1+\mu t)=\sum_{r=0}^k t^r e_r(\mu).
\]
\end{lemma}

\begin{proof}
Since $t\in\F_{q_0}$ we have $\sigma(t)=t$, hence
\[
  \Norm(1+\mu t)=\prod_{j=0}^{k-1}\bigl(1+\sigma^j(\mu)\,t\bigr).
\]
Expanding the product and collecting terms according to the power of $t$ yields the stated identity,
with the coefficient of $t^r$ equal to the $r$th elementary symmetric polynomial in the conjugates $\sigma^j(\mu)$.
\end{proof}

Let $S \subset \F_{q_0}^{k - 1}$ consist of the $(k - 1)$-tuples $(t_1, \cdots, t_{k - 1})$ with pairwise distinct entries none of which being in $\{0, 1\}$. Let $t_k = 1$. We now define a key map $\varphi: S \to \F_{q_0}^k$. For each $(t_1, \cdots, t_{k - 1}) \in S$, let
$$A_i := \prod_{j\neq i}\frac{-t_j}{t_i-t_j}=\prod_{j\neq i}\frac{t_j}{t_j-t_i}\qquad (i=1,\dots,k)$$
and define 
$$\varphi(t_1, \cdots, t_{k - 1}) = (A_1, \cdots, A_k).$$
This definition is motivated by the following lemma.
\begin{lemma}\label{lem:lagrange-weights-k}
Let $(t_1,\dots,t_{k - 1})$ be an element of $S$, and let $(A_1, \cdots, A_k) = \varphi(t_1, \cdots, t_{k-1})$. Recall that $t_k = 1$. Then we have
\[
  \sum_{i=1}^k A_i =1
  \qquad\text{and}\qquad
  \sum_{i=1}^k A_i t_i^r =0\ \ \text{for every}\ \ 1\le r\le k-1.
\]
Moreover, we have
\[
  \sum_{i=1}^k A_i t_i^k = (-1)^{k+1}\prod_{i=1}^k t_i.
\]
\end{lemma}

\begin{proof}
Let $L_i(T)$ be the Lagrange basis polynomials for the nodes $t_1,\dots,t_k$,
\[
  L_i(T):=\prod_{j\neq i}\frac{T-t_j}{t_i-t_j},
\]
so that any polynomial $f$ of degree at most $k-1$ satisfies $f(T)=\sum_{i=1}^k f(t_i)L_i(T)$.
Evaluating at $T=0$ gives
\[
  f(0)=\sum_{i=1}^k f(t_i)L_i(0).
\]
Since $A_i=L_i(0)$, taking $f(T)\equiv 1$ gives $\sum_i A_i=1$, while taking $f(T)=T^r$ for $1\le r\le k-1$ gives $\sum_i A_i t_i^r=0$.

For the final identity, let $g(T)$ be the (unique) interpolation polynomial of degree at most $k-1$
such that $g(t_i)=t_i^k$ for all $i \in [k]$.
Then the polynomial $T^k-g(T)$ has roots $t_1,\dots,t_k$, hence
\[
  T^k-g(T)=\prod_{i=1}^k (T-t_i).
\]
Evaluating at $T=0$ gives $-g(0)=(-1)^k\prod_i t_i$, i.e.\ $g(0)=(-1)^{k+1}\prod_i t_i$.
Finally, by interpolation at $T=0$ we have $g(0)=\sum_{i=1}^k A_i t_i^k$, completing the proof.
\end{proof}

\begin{corollary}\label{cor:norm-cancel-k}
With the notation above, for every $\mu\in\F_q$ we have
\[
  \sum_{i=1}^k A_i\,\Norm(1+\mu t_i)=1 + c\,\Norm(\mu),
\]
where $c=\sum_{i=1}^k A_i t_i^k=(-1)^{k+1}\prod_{i=1}^k t_i\in \F_{q_0}^\times$.
\end{corollary}

\begin{proof}
Combine Lemma~\ref{lem:norm-expansion-k} with Lemma~\ref{lem:lagrange-weights-k}.
All intermediate coefficients $e_r(\mu)$ for $1\le r\le k-1$ vanish after weighting, leaving only the constant term $e_0(\mu)=1$
and the top term $e_k(\mu)=\Norm(\mu)$.
\end{proof}

\smallskip

We now show how Corollary~\ref{cor:norm-cancel-k} leads to an induced matching.
Define the \emph{$k$-fold norm hypersurface} as
\[
  \mathcal{H}_k:=\Bigl\{(x_1,\dots,x_k)\in\F_q^k:\ \sum_{i=1}^k \Norm(x_i)=1\Bigr\}.
\]

\begin{lemma}
\label{lem:unique-line-on-Hk}
Let $(t_1,\dots,t_{k - 1}) \in S$ and let $(A_1, \cdots, A_{k - 1}) = \varphi(t_1, \cdots, t_k)$. Let $(a_1,\dots,a_k)\in\F_q^k$ satisfy
\[
  \Norm(a_i)=A_i\qquad (i=1,\dots,k).
\]
Then the affine line
\[
  \ell_{a_1,\dots,a_k}:=\{(a_1(1+\mu t_1),\dots,a_k(1+\mu t_k)):\ \mu\in\F_q\}
\]
satisfies $\ell_{a_1,\dots,a_k}\cap \mathcal{H}_k=\{(a_1,\dots,a_k)\}$.
\end{lemma}

\begin{proof}
Since $\sum_i A_i=1$ by Lemma~\ref{lem:lagrange-weights-k}, we have $(a_1,\dots,a_k)\in\mathcal{H}_k$.
For $\mu\in\F_q$, by multiplicativity of $\Norm$ and Corollary~\ref{cor:norm-cancel-k},
\[
  \sum_{i=1}^k \Norm\bigl(a_i(1+\mu t_i)\bigr)
  = \sum_{i=1}^k \Norm(a_i)\,\Norm(1+\mu t_i)
  = \sum_{i=1}^k A_i\,\Norm(1+\mu t_i)
  = 1 + c\,\Norm(\mu).
\]
Because $c\neq 0$ and $\Norm(\mu)=0$ if and only if $\mu=0$, the right-hand side equals $1$ precisely when $\mu=0$.
Thus $\ell_{a_1,\dots,a_k}$ meets $\mathcal{H}_k$ only at $\mu=0$, i.e.\ only at $(a_1,\dots,a_k)$.
\end{proof}

Therefore, the subset of $\mathcal{H}_k$ defined by
\[
  P := \{(a_1, \cdots, a_k) \in \F_q^k: (N(a_1), \cdots, N(a_k)) \in \varphi(S)\}
\]
together with the lines $\ell_{a_1, \cdots, a_k}$ forms an induced point--line matching. It remains to show that the matching is large, which is equivalent to lower bounding the size of $\varphi(S)$. 

We show that 

\begin{lemma}
    \label{lem:preimage-size}
    For each $A = (A_1, \cdots, A_k) \in \varphi(S)$, its preimage $\varphi^{-1}(A)$ has size at most $(k - 1)!$.
\end{lemma}

In order to prove this lemma, we need to following ``overdetermined" form of Bez\'{o}ut's theorem. The theorem was first explicitly stated by Heintz in \cite{Heintz}, and a self-contained proof was published by Tao in his webblog \cite{Tao2011}.

\begin{theorem}[Bez\'out's theorem]
    \label{thm:bezout}
    Assume $d \geq m \geq 0$. Let $k$ be a field. Let $f_1, \cdots, f_d \in k[x_1, \cdots, x_m]$ be polynomials, and let $V$ be their common zero locus. Let $V^0$ be the union of the $0$-dimensional irreducible components of $V$. If $f_i$ has degree $D_i$, then we have
    $$|V^0| \leq D_1 \cdots D_d.$$
\end{theorem}
\begin{proof}[Proof of Lemma~\ref{lem:preimage-size}]

Recall that $A = (A_1, \cdots, A_k)$ lies in $\varphi(S)$. 

Let $f_1, \cdots, f_{k - 1} \in \F_{q_0}[x_1, \cdots, x_{k - 1}]$ be defined by
$$f_i(x_1, \cdots, x_{k - 1}) = A_1 x_1^{i} + \cdots + A_{k - 1}x_{k - 1}^i + A_k.$$
By Lemma~\ref{lem:lagrange-weights-k}, any $(t_1, \cdots, t_{k - 1}) \in \varphi^{-1}(A)$ lies in the common zero locus $V$ of the $f_i$'s. We now claim that it lies in a $0$-dimensional irreducible component. 

By the Jacobian criterion, it suffices to check that the Jacobian $J = \{\partial_j f_i\}_{i, j = 1}^{k - 1}$ is non-singular at $(t_1, \cdots, t_{k - 1})$. Note that
$$\partial_j f_i(t_1, \cdots, t_k) = A_j \cdot i \cdot t_{j}^{i - 1}.$$
Hence we have
$$\det J = (k - 1)! \cdot \prod_{j = 1}^{k - 1} A_j \prod_{1 \leq i < j \leq k} (t_j - t_i).$$
As $q_0 > k$, we have $k! \neq 0$. By the definition of $\varphi$, the $A_i$'s are nonzero. By the definition of $S$, the $t_i$'s are nonzero. We conclude that $\det J \neq 0$, thus any element in $\varphi^{-1}(A)$ lies in a $0$-dimensional irreducible component of $V$. By Theorem~\ref{thm:bezout}, we conclude that $\abs{\varphi^{-1}(A)} \leq (k - 1)!$.

\end{proof}

\begin{proof}[Proof of Theorem~\ref{thm:small-power}]
Set
$$
  P := \{(a_1, \cdots, a_k) \in \F_q^k: (N(a_1), \cdots, N(a_k)) \in \varphi(S)\}
$$
Recalling the definition of $S$, we have
$$\abs{S} = (q_0 - 2) (q_0 - 3) \cdots (q_0 - k) \gg_k q_0^{k - 1}.$$
By Lemma~\ref{lem:preimage-size}, we have
$$\abs{\varphi(S)} \gg_k \abs{S} \gg_k q_0^{k - 1}.$$
Since the norm map $\Norm:\F_q^\times\to \F_{q_0}^\times$ is surjective with fibres of size $(q-1)/(q_0-1)$ (see \cite[Ch.~2]{LidlNiederreiter1997}), we obtain
$$\abs{P} = \left(\frac{q - 1}{q_0 - 1}\right)^{k}\abs{\varphi(S)} \gg_k \frac{1}{q_0} q^{k} = q^{k - 1/k}.$$
For each $p=(a_1,\dots,a_k)\in P$ choose the associated line $\ell_p$ from Lemma~\ref{lem:unique-line-on-Hk}.
Since $\ell_p\cap\mathcal{H}_k=\{p\}$ and $P\subset\mathcal{H}_k$, we have $\ell_p\cap P=\{p\}$.
Thus $\{(p,\ell_p):p\in P\}$ is an induced matching in $\II_q^{(k)}$ of size $|P|\gg_k q^{k-1/k}$.

For $d>k$ we take the Cartesian product with $\F_q^{d-k}$:
replace each $p\in P$ by all points $(p,u)\in\F_q^k\times\F_q^{d-k}=\F_q^d$ and replace $\ell_p$ by $\ell_p\times\{u\}$.
This multiplies the matching size by $q^{d-k}$ and yields an induced matching in $\II_q^{(d)}$ of size
\[
  \gg_k q^{k-1/k}q^{d-k} = q^{d-1/k},
\]
as desired.
\end{proof}

\begin{remark}[The case $k=2$]
When $k=2$ and $q=q_0^2$, the hypersurface $\mathcal{H}_2=\{(x_1,x_2):\Norm(x_1)+\Norm(x_2)=1\}$ is a Hermitian curve,
and lines meeting $\mathcal{H}_2$ in a unique point are its tangents.
Thus, for $d=2$ the norm--interpolation construction is closely related to the classical Hermitian unital and yields
$\IM(2,q)\gg q^{3/2}$, matching \eqref{eq:vinh} up to constants in the square-field case.
\end{remark}

\section{Ruzsa lift for prime powers} \label{sec:Ruzsa-prime-powers}

In this section, we finally complete the proof of Theorem~\ref{thm:prime-power-high-dim}. We will prove the following result, which deals with the prime powers $q = p^t$ with $t > d$. The remaining prime powers are covered by Corollary~\ref{cor:prime-power} and Theorem~\ref{thm:small-power}. 
\begin{lemma}
\label{lem:prime-power-high-dim-final-piece}
Let $k$ be an odd positive integer such that $2k + 1$ is a prime. Then there exists some constant $f(k) = k^{O(k)}$ such that the following holds. Let $d, t$ be positive integers with $\min(d, t) > f(k)$, and let $q = q_0^t$ be a prime power. Then we have 
$$\IM(d, q) \geq  (c_k)^t t^{-t} \cdot q^{d - \delta_k}$$
where $c_k > 0$ depends on $k$ only and $\delta_k = \frac{3}{\log k} + \frac{3}{k}$.
\end{lemma}
\begin{proof}[Proof of Theorem~\ref{thm:prime-power-high-dim} assuming Lemma~\ref{lem:prime-power-high-dim-final-piece}]
We may assume that $d$ is sufficiently large. For a prime power $q = p^t$, we divide into three cases based on the exponent $t$.
\begin{enumerate}
    \item If $t < \sqrt{\log d}$, then Corollary~\ref{cor:prime-power} gives
    $$\IM(d, q) \gg_{d, t} q^{d - \epsilon_{d, t}}$$
    where $\epsilon_{d, t} \ll t / \log d \ll (\log d)^{-1/2} \ll (\log\log d)^{-1}$, as desired.

    \item If $t \in [\sqrt{\log d}, d]$, then Theorem~\ref{thm:small-power} gives
    $$\IM(d, q) \gg_{d} q^{d - 1 / t} \gg q^{d - (\log d)^{-1/2}}$$
    as desired.

    \item If $t > d$, then Lemma~\ref{lem:prime-power-high-dim-final-piece} gives
    $$\IM(d, q) \geq (c_k)^t t^{-t} \cdot q^{d - \delta_k}$$
    where $k$ is the maximum odd positive integer such that $(2k + 1)$ is a prime and $f(k) < d$. As $f(k) = k^{O(k)}$, we have $k \gg \sqrt{\log d}$, hence $\delta_k \ll (\log\log d)^{-1}$. Furthermore, we have
    $$(c_k)^t t^{-t} \ll_k t^{-2t} = q^{2\log t / \log p}.$$
    Hence, taking $\epsilon_d = \delta_k$, we conclude that
    $$\IM(d, q) \gg_d q^{d - \epsilon_d - 2 \log t / \log p}.$$
\end{enumerate}
Thus, we have proved Theorem~\ref{thm:prime-power-high-dim} in all cases.
\end{proof}
Our approach to Lemma~\ref{lem:prime-power-high-dim-final-piece} is an analogue to the proof of Theorem~\ref{thm:prime-high-dim} with $\Z$ replaced by the polynomial ring $\Z[T]$ (in this proof, we use $T$ as the variable for polynomials). We need the following analogue of Waring's theorem for polynomials (see \cite{CHWaring, Zhu}). For the reader's convenience, we include a proof in the Appendix.

\begin{proposition}\label{prop:shifted-kth-power-sum}
Let $k\ge 1$ and $f(x)\in \mathbb{Z}[T]$. then
\[
\sum_{i = 0}^{k - 1} (-1)^i \binom{k - 1}{i} \left(f + \frac{k-1}{2} - i\right)^k = k! \cdot f.
\]
\end{proposition}

We set $\ell = (k + 1)!$ and define an index set
$$J := J_k = \{(\alpha, \beta, \beta', \gamma): \alpha \in [k - 1], \beta, \beta' \in [0, \ell (k - \alpha)], \gamma \in [0, \alpha - 1]\} \cup \{0\}.$$
It is clear that $\abs{J_k} = k^{O(k)}$. We take $f(k) = \max(\ell^3, \abs{J_k} + 1)$. Then the condition $\min(d, t) > f(k)$ implies that
$$t > \ell^3, d > f(k) + 1.$$
Let $r = \lceil t / \ell \rceil$. As in the proof of Proposition~\ref{prop:ruzsa-lift-prime-power}, let $X(s, M)$ denote the polynomials in $\Z[T]$ with degree at most $(s - 1)$, whose coefficients have absolute value at most $M$. We note the relation that
\begin{equation}
    \label{eq:X-product}
 X(s, M) \cdot X(s', M') \subset X(s + s', (s + s') M M')   .
\end{equation}
Let $N$ be a positive integer and set $M = \lfloor N^{1 / \ell} \rfloor$.  We define the following polynomial
$$\Psi_{N,t}(x_J) := (k!)^{k} x_0^k + \sum_{\alpha \in [k - 1]}\sum_{\beta, \beta' \in [0, \ell (k - \alpha)]}(k!)^{\alpha} M^{\beta} T^{r \beta'} \left(\sum_{\gamma = 0}^{\alpha - 1} (-1)^{\gamma} \binom{\alpha - 1}{\gamma} x_{\alpha, \beta, \beta', \gamma}^{\alpha}\right).$$
This polynomial satisfies the following analog of Lemma~\ref{lem:nice--line} in the ring $\Z[T]$.
\begin{lemma}
    \label{lem:nice--line--poly}
    For any $x_J \in X(t, N)^J$, there exists some $y_J \in \Z[T]^J$ such that, as a polynomial identity in the single variable $h$, we have
    $$\Psi_{N,t}(x_J + y_J h) = \Psi_{N,t}(x_J) + (k!)^k \cdot h^k,$$
    and furthermore, we have $y_j \in X(t', N')$ for every $j \in J$, where $t' \leq 2k! \cdot (r + (k + 1)!)$, $N' = O_k(N^{1 / k})$.
\end{lemma}
\begin{proof}
    For each $\alpha \in \{1, \cdots, k - 1\}$. Let $J_{\alpha} = \{(\alpha, \beta, \beta', \gamma):\beta,\beta'\in [0,\ell(k-\alpha)], \gamma\in [0,\alpha-1]\}$ and $y_{\alpha} = y_{J_{\alpha}}$. Note that the term in $\Psi_{N,t}(x_J + y_J h)$ with index $(\alpha, \beta, \beta', \gamma)$ only contribute to the coefficients of $h^{\alpha'}$ with $\alpha' \leq \alpha$. 
    
    We set $y_0 = 1$, and also define $t_k := 1, N_k := 1$. We then iteratively choose $y_{\alpha} \in \Z[T]^{J_{\alpha}}$ and integers $t_{\alpha}, N_{\alpha}$ in decreasing order of $\alpha = k - 1 , \cdots, 1$, such that the coefficient of $h^{\alpha}$ in $\Psi_{N,t}(x_J + y_J h) $ is zero, and $y_{j} \in X(t_{\alpha}, N_{\alpha})$ for each $j \in J_{\alpha}$.
    
    This coefficient is given by
    $$(k!)^{k} \binom{k}{\alpha} x_0^{k - \alpha} + \sum_{\alpha' = \alpha}^{k - 1}\sum_{\beta, \beta' \in [0, \ell (k - \alpha')]}(k!)^{\alpha'} M^{\beta} T^{r \beta'} \left(\sum_{\gamma = 0}^{\alpha' - 1} (-1)^{\gamma} \binom{\alpha' - 1}{\gamma}  \binom{\alpha'}{\alpha}x_{\alpha', \beta, \beta', \gamma}^{\alpha' - \alpha} y_{\alpha', \beta, \beta', \gamma}^{\alpha} \right).$$
    Isolating the term corresponding to $I_{\alpha}$, we need to ensure that
    $$\sum_{\beta, \beta' \in [0, \ell (k - \alpha)]} (k!)^{\alpha} M^{\beta} T^{r \beta'} \left(\sum_{\gamma = 0}^{\alpha - 1} \binom{\alpha - 1}{\gamma} (-1)^{\gamma}  y_{\alpha, \beta, \beta', \gamma}^{\alpha}\right) = R$$
    where
    $$R = - (k!)^{k} \binom{k}{\alpha} x_0^{k - \alpha} - \sum_{\alpha' = \alpha + 1}^{k - 1}\sum_{\beta, \beta' \in [0, \ell (k - \alpha')]}(k!)^{\alpha'} M^{\beta} T^{r \beta'} \left(\sum_{\gamma = 0}^{\alpha' - 1} (-1)^{\gamma} \binom{\alpha' - 1}{\gamma}  \binom{\alpha'}{\alpha}x_{\alpha', \beta, \beta', \gamma}^{\alpha' - \alpha} y_{\alpha', \beta, \beta', \gamma}^{\alpha} \right).$$
    We now recall our assumption that
    $$x_{\alpha', \beta, \beta', \gamma} \in X(t, N), y_{\alpha', \beta, \beta', \gamma} \in X(t_{\alpha'}, N_{\alpha'}).$$
    Applying \eqref{eq:X-product}, we have
    $$x_{\alpha', \beta, \beta', \gamma}^{\alpha' - \alpha} y_{\alpha', \beta, \beta', \gamma}^{\alpha} \in X\left((\alpha' - \alpha) t + \alpha t_{\alpha'}, O_k\left(N^{\alpha' - \alpha} N_{\alpha'}^{\alpha}\right)\right).$$
    Therefore, for each $\alpha' \in [\alpha, k - 1]$, the summand
    $$\sum_{\beta, \beta' \in [0, \ell (k - \alpha')]}(k!)^{\alpha'} M^{\beta} T^{r \beta'} \left(\sum_{\gamma = 0}^{\alpha' - 1} (-1)^{\gamma} \binom{\alpha' - 1}{\gamma}  \binom{\alpha'}{\alpha'}x_{\alpha', \beta, \beta', \gamma}^{\alpha' - \alpha'} y_{\alpha', \beta, \beta', \gamma}^{\alpha} \right)$$
    lies in 
    $$X\left((\alpha' - \alpha) t + \alpha t_{\alpha'} + r \ell (k - \alpha'), O_k\left(M^{\ell (k - \alpha')}N^{\alpha' - \alpha} N_{\alpha'}^{\alpha}\right)\right).$$
    Furthermore, we have
    $$- (k!)^{k} \binom{k}{\alpha} x_0^{k - \alpha} \in X(t (k - \alpha), N^{k - \alpha}).$$
    Recalling the assumption that $r = \lceil {t / \ell} \rceil \leq t / \ell + 1$ and $M = \lfloor N^{1 / \ell} \rfloor$, we conclude that
    $$R \in X\left((k - \alpha) t + \ell k + \alpha \max_{\alpha' > \alpha} t_{\alpha'}, O_k\left(N^{k - \alpha} \max_{\alpha' > \alpha} N_{\alpha'}^{\alpha}\right)\right).$$
    We also make the observation that $(k!)^{-\alpha - 1} R$ has integral coefficients.
    
    We now write the coefficients of  $(k!)^{-\alpha - 1} R$ in base $M$, and write each monomial $T^{i}$ as $T^{(i - r \beta') + r \beta'}$, where $\beta'$ is the largest element of $ [0, \ell (k - \alpha)]$ such that $i \geq r \beta'$. Thus, there exists $R_{\beta, \beta'} \in \Z[T]$ such that
    $$R = \sum_{\beta, \beta' \in [0, \ell (k - \alpha)]}(k!)^{\alpha + 1} M^{\beta} T^{r \beta'} R_{\beta, \beta'}$$
    and as $r \ell (k - \alpha) \geq t (k - \alpha)$ and $M^{\ell (k - \alpha)} \geq 2^{-k} N^{k - \alpha}$, we have
    $$R_{\beta, \beta'} \in X\left(\max(r, \ell k + \alpha \max_{\alpha' > \alpha} t_{\alpha'}), \max(M, O_k(\max_{\alpha' > \alpha} N_{\alpha'}^{\alpha}))\right).$$
    Finally, recalling our assumption that $k$ is odd, we can take
    $$y_{\alpha, \beta, \beta', \gamma} = R_{\beta, \beta'} + \frac{k - 1}{2} - \gamma$$
    so that by Proposition~\ref{prop:shifted-kth-power-sum}, we have
    $$k! R_{\beta, \beta'} = \sum_{\gamma = 0}^{\alpha' - 1} \binom{\alpha - 1}{\gamma} (-1)^{\gamma}  y_{\alpha, \beta, \beta', \gamma}^{\alpha}$$
    and thus we conclude that
    $$R = \sum_{\beta, \beta' \in [0, \ell (k - \alpha)]}(k!)^\alpha M^{\beta} T^{r \beta'} \sum_{\gamma = 0}^{\alpha' - 1} \binom{\alpha - 1}{\gamma} (-1)^{\gamma}  y_{\alpha, \beta, \beta', \gamma}^{\alpha}$$
    as desired. Furthermore, we have
    $$y_{\alpha, \beta, \beta', \gamma} \in X\left(\max(r, \ell k + \alpha \max_{\alpha' > \alpha} t_{\alpha'}), k + \max(M, O_k(\max_{\alpha' > \alpha} N_{\alpha'}^{\alpha}))\right).$$
    Thus, we can take
    $$\begin{cases}
        t_{\alpha} = r + \ell k + \alpha \max_{\alpha' > \alpha} t_{\alpha'} \\ 
        N_{\alpha} \ll_k M + \max_{\alpha' > \alpha} N_{\alpha'}^{\alpha}
    \end{cases}.$$
    Solving this recursion, we conclude that for each $\alpha$, we have
    $$t_{\alpha} \leq t' = 2k! \cdot (r + \ell k)$$
    and
    $$N_{\alpha} \ll_k M^{(k - 1)!} \ll_k N^{(k - 1)! / \ell}.$$
    Recalling our choice that $\ell = (k + 1)!$, we conclude the desired result.
\end{proof}
We are now ready to construct the induced point--line matching over $\IM(d, q)$. In light of \eqref{eq:IMlift}, and recall that we assume $d > |J_k| + 1$, it suffices to show that
$$\IM(|J_k| + 1, q) \geq (c_k)^t t^{-t} q^{|J_k| + 1 - 4 (\log k)^{-1}}.$$
Recall our choice of parameters:
$$\ell = (k + 1)!, \quad t > \ell^3, \quad r = \lceil t / \ell \rceil.$$

Let $N = \lfloor{p / 4\rfloor}$. Let $A \subset [N^k]$ be the largest subset of $[N^k]$ such that $A - A$ avoids non--zero $k$--th powers. Let $Y_k(r \ell k, N^k, A)$ denote the set of polynomials in $X(r \ell k, N^k)$, whose coefficient of $T^{ki}$ lies in $A$ for each non--negative integer $i$.  

For a shift $s \in \Z[T]$, we define the set
$$\Gamma_s = \{x_J \in X(t, N)^{J}: \Psi_{N,t}(x_J) \in Y_k(r \ell k, N^k, A) + s\}.$$
Recall the definition of $\Psi_{N,t}$ as
$$\Psi_{N,t}(x_J) = (k!)^{k} x_0^k + \sum_{\alpha \in [k - 1]}\sum_{\beta, \beta' \in [0, \ell (k - \alpha)]}(k!)^{\alpha} M^{\beta} T^{r \beta'} \left(\sum_{\gamma = 0}^{\alpha - 1} (-1)^{\gamma} \binom{\alpha - 1}{\gamma} x_{\alpha, \beta, \beta', \gamma}^{\alpha}\right)$$
where $M = \lfloor{N^{1 /\ell}\rfloor}$. By \eqref{eq:X-product}, we see that for any $x_J \in X(t, N)^{J}$, we have
$$\Psi_{N,t}(x_J) \in X\left(r \ell k, O_k(N^k)\right)$$
where the bound on the degree following from that fact that for any $\alpha \in [k - 1]$, we have
$$\ell (k - \alpha) r + \alpha t \leq r \ell k.$$
Hence, $\abs{\Gamma_s} = 0$ unless $s \in X(r \ell k, C_k N^k)$ for some constant $C_k > 0$ depending on $k$ only. On the other hand, we have
$$\sum_{s \in \Z[T]} \abs{\Gamma_s} = \abs{X(t, N)}^{\abs{J}} \abs{Y_k(r \ell k, N^k, A)}.$$
Therefore, we can choose some $s \in \Z[T]$ such that
$$\abs{\Gamma_s} \geq \frac{\abs{X(t, N)}^{\abs{J}}\abs{Y_k(r \ell k, N^k, A)}}{\abs{X(r \ell k, C_k N^k)}}.$$
By Lemma~\ref{lem:nice--line--poly}, for any $x_J \in X(t, N)^J$, there exists some $y_J \in \Z[T]^J$ such that (for every $h\in\Z[T]$)
$$\Psi_{N,t}(x_J + y_J h) = \Psi_{N,t}(x_J) + (k!)^{k}h^k$$
such that for each $j \in J$, we have
$$y_j \in X(t', N')$$ 
with $t' = 2k! \cdot (r + (k + 1)!)$, $N' = D_k N^{1 / k}$ and $D_k > 0$ is some constant depending only on $k$.

Now let $\alpha$ be a primitive element of the field extension $\F_q / \F_p$.
We consider the subset of $\F_q \times \F_q^J$ defined by\footnote{Here $(h, x_J)(\alpha)$ means evaluating every coordinate of $(h, x_J)$, being a polynomial in $T$, at $T = \alpha$.}
$$P := \{(z, x_J)(\alpha): z \in X(t - t', N / 2tN'), x_J \in \Gamma_s\}.$$
For each $p = (z, x_J)(\alpha)$ in $P$, define the line $\ell_p = \{(z(\alpha) + h, x_J(\alpha) + y_J(\alpha) h): h \in \F_q\}$. We claim that the reduction of $P$ and $\{\ell_p\}$ modulo $q$ is an induced point--line matching.

Suppose for the sake of contradiction that for some distinct points $p = (z, x_J)(\alpha)$ and $p' = (z', x_J')(\alpha)$ in $P$, we have $p' \in \ell_{p}$. Let $h = z' - z$. Then we have $h \in X(t - t', N / tN')$. For each coordinate $j \in J$, we have
$$x'_j(\alpha) = x_j(\alpha) + h y_j(\alpha).$$
Note that each of the polynomials $x_j'$, $x_j$, $hy_j$ lies in $X(t, N)$. As the minimal polynomial of $\alpha$ over $\F_p$ has degree $t$, we must have
$$x_j' - x_j - h y_j \equiv 0 \bmod{p}.$$
Furthermore, the coefficients of $x_j' - x_j - h y_j$ have absolute value at most $3N < p$. Hence, we must have
$$x_j' - x_j - h y_j  = 0$$
over $\Z[T]$.
Therefore, we have
$$\Psi_{N,t}(x_J') = \Psi_{N,t}(x_J + hy_J).$$
Expanding the right hand side using Lemma~\ref{lem:nice--line--poly}, we have
$$\Psi_{N,t}(x_J') = \Psi_{N,t}(x_J) + (k! h)^k.$$
On the other hand, $\Psi_{N,t}(x_J') - \Psi_{N,t}(x_J)$ lies in $Y_k(r \ell k, N^k, A) - Y_k(r \ell k, N^k, A)$. The non--zero term of $(k! h)^k$ with the least degree must be of the form $s^k T^{ki}$ for some non--negative integer $i \geq 0$. In any element of $Y_k(r \ell k, N^k, A) - Y_k(r \ell k, N^k, A)$, the coefficient of $T^{ki}$ must lie in $A - A$, so it cannot be a perfect $k$--th power, contradiction.

We conclude that
$$P = \{(z, x_J)(\alpha): z \in X(t - t', N / 2tN'), x_J \in \Gamma_s\}$$
together with the lines $\ell_p$ we defined forms an induced point--line matching in $\F_q^{|J| + 1}$. 

Different choices of $(z, x_J)$ give rise to distinct points of $P$, since $z$ and $x_J$ have degree less than $t$ and coefficients less than $p / 2$ in absolute value. Therefore, we have
$$\abs{P} \geq \abs{X(t - t', N / 2tN')} \abs{\Gamma_s}.$$
Substituting our estimate for $\abs{\Gamma_s}$, we have
$$\abs{P} \gg_k \abs{X(t - t', N / 2tN')} \cdot \frac{\abs{X(t, N)}^{\abs{J}}\abs{Y_k(r \ell k, N^k, A)}}{\abs{X(r \ell k, C_k N^k)}}.$$
It is clear that for any $K \geq 0$, we have $\max(K, 1)^t \leq \abs{X(t, K)} \leq (2K + 1)^t$. Hence we obtain
$$\abs{P} \geq c_k^{t} \cdot (N / 2tN')^{t - t'} \cdot \frac{N^{t\abs{J}}\abs{Y_k(r \ell k, N^k, A)}}{N^{r \ell k^2}}$$
where $c_k > 0$ depends only on $k$ and might be different for each appearance.

Furthermore, by definition we have 
$$\abs{Y_k(r \ell k, N^k, A)} \geq N^{k \cdot r \ell (k - 1)} \cdot \abs{A}^{r \ell}.$$
Hence, we obtain
$$\abs{P} \geq c_k^{t} \cdot (N / 2tN')^{t - t'} \cdot \frac{N^{t\abs{J}}\cdot N^{k \cdot r \ell (k - 1)} \cdot \abs{A}^{r \ell}}{N^{r \ell k^2}}.$$
Simplifying, we get
$$\abs{P} \geq c_k^{t} \cdot N^{-t'} (2tN')^{-t} \left(\frac{|A|}{N^k}\right)^{ r \ell} \cdot N^{t (\abs{J} + 1)}.$$
Again, recall our choice of parameters
$$\ell = (k + 1)!, \quad t > \ell^3, \quad r = \lceil t / \ell \rceil, \quad t' = 2k! \cdot (r + (k + 1)!), \quad N' = N^{1 / k}$$
and $|A| = \Omega_k(N^{1 - \frac{2}{k \log k}})$. We have
$$N^{t'}  \leq N^{2t / k},$$
$$(N')^{t} \leq N^{t / k},$$
$$\left(\frac{|A|}{N^k}\right)^{ r \ell} \leq N^{-2 r\ell / \log k} \leq N^{-3t / \log k}.$$
So we can finally conclude that
$$\abs{P} \geq (c_k)^t t^{-t} \cdot N^{t (|J| + 1 - \delta_k)}.$$
Recalling that $N = \lfloor p / 4 \rfloor$ and $q = p^t$, we obtain
$$\abs{P} \geq (c_k)^t t^{-t} \cdot q^{|J| + 1 - \delta_k}$$
as desired.

\section{New Nikodym sets and new minimal blocking sets}\label{sec:Nikodym-I}

In this section, we will first justify the simple correspondence between induced matchings in point--line incidence graphs and (weak) Nikodym sets, which was discussed in Section \ref{sec:intro}. In dimensions $d\ge 3$, this dictionary immediately turns any large induced matching in dimension $(d - 1)$ into a small Nikodym set by taking a
Cartesian product. In dimension $2$, the same dictionary only gives \emph{weak} Nikodym sets. In order to prove Theorem \ref{thm:nikodym-2d}, we will start instead with a high-dimensional matching constructed in Section~\ref{sec:prime-highdim}, and then develop a projection mechanism that will produce a Nikodym set in $\F_{q}^{2}$. 

Last but not least, we will use the new Nikodym sets we construct to get new minimal blocking sets, establishing Theorem~\ref{thm:blocking-2d-poly}. 

\subsection{Small Nikodym sets in dimension $d\ge 3$}

Recall that for $x\in\F_q^d$ and $v\in\F_q^d\setminus\{0\}$ we write
\[
  \ell(x,v):=\{x+\lambda v:\lambda\in\F_q\}
  \qquad\text{and}\qquad
  \ell(x,v)^\ast:=\ell(x,v)\setminus\{x\}.
\]
A set $N\subset \F_q^d$ is a \emph{Nikodym set} if for every $x\in\F_q^d$ there exists $v\neq 0$ with
$\ell(x,v)^\ast\subset N$.  A set $N\subset \F_q^d$ is a \emph{weak Nikodym set} if for every $x\notin N$ there exists
$v\neq 0$ with $\ell(x,v)^\ast\subset N$.

\begin{proof}[Proof of Proposition~\ref{prop:nikodym-basic}]
\emph{(1) Nikodym $\Rightarrow$ weak Nikodym, and the converse fails for $d=2$.}
The implication is immediate from the definitions.
To see that the converse can fail in the plane, consider
\[
  N:=(\F_q^\times)^2=\{(x,y)\in\F_q^2:\ x\neq 0,\ y\neq 0\}.
\]
If $p\in\F_q^2\setminus N$, then either $p=(a,0)$ with $a\neq 0$, or $p=(0,b)$ with $b\neq 0$, or $p=(0,0)$.
In these three cases one checks that the punctured line is contained in $N$ by choosing respectively
\[
  v=(0,1),\qquad v=(1,0),\qquad v=(1,1).
\]
Hence $N$ is weak Nikodym.
On the other hand, if $p=(a,b)\in N$ and $\ell$ is any affine line through $p$, then $\ell$ contains a point with
$x$-coordinate $0$ (unless $\ell$ is vertical, in which case it contains a point with $y$-coordinate $0$).
Thus every line through $p$ meets $\F_q^2\setminus N$, so $N$ is not Nikodym.

\medskip
\emph{(2) Weak Nikodym sets and induced matchings are complements.}
Suppose $N\subset\F_q^d$ is weak Nikodym.
For each $x\in\F_q^d\setminus N$, choose a line $\ell_x$ with $\ell_x^\ast\subset N$.
If $x\neq x'$ are two points outside $N$, then $x'\notin \ell_x$ (since $\ell_x^\ast\subset N$ but $x'\notin N$),
so the pairs $\{(x,\ell_x)\}_{x\notin N}$ form an induced matching in the point--line incidence graph.
Conversely, if $\{(p,\ell_p)\}_{p\in M}$ is an induced matching, then with $N:=\F_q^d\setminus M$ we have
$\ell_p^\ast\subset N$ for each $p\in M$, i.e.\ $N$ is weak Nikodym.

\medskip
\emph{(3) Product trick: weak Nikodym $\Rightarrow$ Nikodym one dimension up.}
Assume $N\subset\F_q^d$ is weak Nikodym and set $N':=N\times \F_q\subset \F_q^{d+1}$.
Fix $(x,a)\in\F_q^{d+1}$.
If $x\notin N$, pick $v_x\neq 0$ with $\ell(x,v_x)^\ast\subset N$ and take $v':=(v_x,0)$; then
$\ell((x,a),v')^\ast\subset N\times\{a\}\subset N'$.
If $x\in N$, take $v':=(0,\dots,0,1)$; then $\ell((x,a),v')^\ast=\{(x,a+\lambda):\lambda\neq 0\}\subset N'$.
Thus $N'$ is a Nikodym set in $\F_q^{d+1}$.
\end{proof}

We now derive Theorem~\ref{thm:nikodym-3d}.

\begin{proof}[Proof of Theorem~\ref{thm:nikodym-3d}]
Let $M\subset \F_q^{d - 1}$ be the point set of an induced matching in $\II_q^{(d) - 1}$, 
and set $N_0:=\F_q^{d - 1}\setminus M$.
By Proposition~\ref{prop:nikodym-basic}(2), $N_0$ is weak Nikodym in $\F_q^{d - 1}$.
Applying Proposition~\ref{prop:nikodym-basic}(3) repeatedly, we obtain that
\[
  N:=N_0\times \F_q^{\,d-1}\subset \F_q^d
\]
is a Nikodym set. Its size is
\[
  |N| = (q^{d - 1}-|M|)\,q = q^d - |M|\,q.
\]

We now choose $M$ from the appropriate induced-matching theorem.

\smallskip
\noindent\emph{Three Dimensions.}  Take $d=3$ and let $M\subset\F_q^2$ be the point set of the induced matching from Proposition~\ref{prop:ruzsa-lift-prime-power}. This gives
$$|N| = q^3 - \Omega_t(q^{2.1167}).$$

\smallskip

\smallskip
\noindent\emph{Higher dimensions.}  Let $M\subset\F_q^{d - 1}$ be the point set from Theorem~\ref{thm:prime-power-high-dim}.
This gives
$$|N| = q^d - \Omega_d\left(q^{d - \epsilon_{d - 1}  2 \log t / \log p}\right)$$
where $\epsilon_{d - 1} \ll (\log\log (d-1))^{-1} \ll (\log\log d)^{-1}$.
\end{proof}

\subsection{Nikodym sets in $2$ dimensions}

We now explain how to obtain the \emph{polynomial} improvement for planar Nikodym sets from Theorem \ref{thm:nikodym-2d}. The main idea will be to use our high-dimensional point-line matching from Theorem \ref{thm:prime-power-high-dim} in order to produce a large set of lattice points in $[N]^{d} \times [M]$, each endowed with a suitable ``escape line'' (one should think of this as a certain Nikodym-like property in $\Z^{d+1}$). We will then be able to project such a set to $\F_{q}^{2}$. 

We start by first recording this projection mechanism.

\begin{proposition}\label{prp:higher dim to 2d}
Suppose $L,M,N$ are positive integers with $LM\le N$.
Let $P\subset [N]^d\times [M]$ be a set of points with the following property:
for every $v\in [N]^d\times [M]$ there exists a direction
\[
  s_v\in [-L,L]^d\times\{1\} 
\]
such that the (integer) punctured line $\ell(v,s_v)^\ast:=\{v+t s_v:t\in\Z\setminus\{0\}\}$ is disjoint from $P$.

Then for every prime $q>(3N)^d$ we have
\[
  \Nik(2,q)\ \le\ q^2-|P|.
\]
\end{proposition}

\begin{proof}
Define a map $\phi:\Z^{d+1}\to\F_q^2$ by
\[
\phi(n_1,\dots,n_d,m)
:=\Bigl(\sum_{i=1}^d n_i(3N)^{i-1}\bmod q,\ m\bmod q\Bigr).
\]
Since $q>(3N)^d$, the base-$(3N)$ expansion in the first coordinate is unique on the box
\[
B:=\bigl[-(N-1),\,2N\bigr]^d\times\bigl[-(M-1),\,2M\bigr],
\]
hence $\phi$ is injective on $B$. In particular $\phi|_P$ is injective and therefore
$|\phi(P)|=|P|$.

Set
\[
\mathcal{N}:=\F_q^2\setminus \phi(P).
\]
We claim that $\mathcal{N}$ is a Nikodym set; since $|\mathcal{N}|=q^2-|P|$, this will prove the proposition.

Fix $w=(w_1,w_2)\in \F_q^2$. Let
\[
X:=\Bigl\{\sum_{i=1}^d n_i(3N)^{i-1}\bmod q:\ n_i\in[N]\Bigr\},
\qquad
Y:=\{1,2,\dots,M\}\subset \F_q,
\]
so that $\phi([N]^d\times[M])\subset X\times Y$.

\smallskip\noindent
\textbf{Case 1: $w\notin X\times Y$.}
If $w_1\notin X$, then the vertical line through $w$ is disjoint from $X\times Y$ and hence from $\phi(P)$.
If instead $w_2\notin Y$, then the horizontal line through $w$ is disjoint from $X\times Y$ and hence from $\phi(P)$.
In either case $w$ has a punctured line contained in $\mathcal{N}$.

\smallskip\noindent
\textbf{Case 2: $w\in X\times Y$.}
Then there exists $v\in[N]^d\times[M]$ with $w=\phi(v)$.
Let $z_w:=\phi(s_v)\in\F_q^2$. We claim that $\ell(w,z_w)^\ast\cap \phi(P)=\emptyset$.

Take any $t\in\F_q^\times$.
If $t\not\equiv t_0\pmod q$ for every integer $t_0\in\{-M,\dots,M\}$, then
$w_2+t\notin Y$ (since $Y-Y\subset\{-M,\dots,M\}\bmod q$), and hence $w+t z_w\notin X\times Y$,
so certainly $w+t z_w\notin\phi(P)$.

Otherwise, write $t\equiv t_0\pmod q$ with $t_0\in\{-M,\dots,M\}\setminus\{0\}$.
By linearity of $\phi$ we have
\[
w+t z_w=\phi(v)+t\phi(s_v)=\phi(v+t_0 s_v)=:\phi(v').
\]
Since $|t_0|\le M$ and $\|s_v\|_\infty\le L$ with $LM\le N$, we have $v'\in B$.
Moreover, by hypothesis $\ell(v,s_v)^\ast\cap P=\emptyset$, so $v'\notin P$.
Injectivity of $\phi$ on $B$ then implies $\phi(v')\notin\phi(P)$, i.e. $w+t z_w\notin\phi(P)$.

Thus for every $t\neq 0$ we have $w+t z_w\in \mathcal{N}$, proving that $\mathcal{N}$ is Nikodym.
\end{proof}

\medskip

We next explain how to produce sets $P\subset[N]^d$ with many points and many ``private'' lines, as required above.
The following statement can be extracted from the high-dimensional induced-matching construction in Section~\ref{sec:prime-highdim}.

\begin{proposition}\label{prop:lattice matching}
For every $\varepsilon>0$ there exists $d=O_\varepsilon(1)$ such that for every $N\ge 1$ one can find a set
$P\subset [N]^d$ with
\[
  |P|\ \gg_\varepsilon\ N^{d-\varepsilon},
\]
and such that for each $p\in P$ there exists a non-zero direction $s_p\in [-N^{\varepsilon},N^\varepsilon]^d$ with
\[
  \ell(p,s_p)^\ast\cap P=\emptyset.
\]
Moreover, one may take $d\le \exp(O(\varepsilon^{-1}))$.
\end{proposition}

\begin{proof}
 We recall the proof of Theorem~\ref{thm:prime-high-dim} in Section~\ref{sec:prime-highdim}. Let $k$ be a positive integer so that $2k+1$ is prime and $\frac{2}{\log k}+\frac{1}{k^2}\le \varepsilon$ (thus $k\le \exp(O(\varepsilon^{-1}))$), and let $I = I_k$ be the index set introduced in Section~\ref{sec:prime-highdim}. Set $d = |I_k| + 1$ (thus $d\le k^{O(1)}\le \exp(O(\varepsilon^{-1}))$).
 
 Observe that we may assume $N$ is sufficiently large, by changing the implicit constant in $|P|\gg_\varepsilon N^{d-\varepsilon}$. At the price of further constants, we may assume $N=\lfloor q/4\rfloor$ for some prime $q$ (we do this only to match the notation used to prove Theorem~\ref{thm:prime-high-dim}).
 
 Now, for each large prime $q$, write $N = \lfloor{q / 4\rfloor}$ and $M = \lfloor{ N^{1 / k^2} \rfloor}$. We constructed a set $\Gamma_s \subset [N]^{I}$, and then considered a subset of $[q] \times [q]^{d - 1} $ defined by
    $$P := \{(z, x_I) \in [q] \times [q]^I: z \in [N / C_k M], x_I \in \Gamma_s\}$$
    with size
    $$\abs{P} \gg_k q^{d - \frac{2}{\log k} - \frac{1}{k^2}}\gg_\varepsilon N^{d-\varepsilon}.$$
    For each $p = (z, x_I)$ in $P$, we define the line $\ell_p = \{(z + h, x_I + y_I h): h \in \Z\}$ with slope $s_p = (1, y_I)$, where $||y_I||_\infty \leq C_k M \le N^{\varepsilon}$ (using that $N$ is large for the final inequality). 
    
    We checked that if $p,p'\in P$ satisfied $p'\equiv p+t\cdot s_p\pmod{q}$ for some $t\in \Z$, that $p=p'$. This in particular implies that $\ell_p\cap P=\{p\}$ or equivalently $\ell(p,s_p)^* \cap P=\emptyset$, as desired.
\end{proof}

We now upgrade Proposition~\ref{prop:lattice matching} to the stronger hypothesis needed in
Proposition~\ref{prp:higher dim to 2d}, namely obtaining an escaping line for \emph{every} point of the ambient box whose slope $s$ has `$1$' as its final coordinate. 

\begin{corollary}\label{cor:lattice-nikodym}
For every $\varepsilon>0$ there exists $d=O_\varepsilon(1)$ such that for all $N\ge 1$, with
\[
M:=N^{1-\varepsilon},\qquad L:=N^\varepsilon,
\]
there exists a set $P\subset [N]^d\times [M]$ with
\[
  |P|\ \gg_\varepsilon\ N^{d+1-2\varepsilon},
\]
and such that for every $v\in [N]^d\times [M]$ there exists $s_v\in [-L,L]^d\times\{1\}$ satisfying
\[
  \ell(v,s_v)^\ast\cap P=\emptyset.
\]

In fact, we may take $d=\exp(O(\varepsilon^{-1}))$.
\end{corollary}

\begin{proof}
Let $d_0$ be large enough so that Proposition~\ref{prop:lattice matching} holds for $\varepsilon$. Now consider $N\ge 1$.
Let $P_0\subset [N]^{d_0}$ be a set of size $|P_0|\gg_\varepsilon N^{d_0-\varepsilon}$ and for each $u\in P_0$ fix a slope 
$s_u^{(0)}\in [ -L,L]^{d_0}$ with $\ell(u,s_u^{(0)})^\ast\cap P_0=\emptyset$.

Define $P_1:=P_0\times [N]\subset [N]^{d_0+1}$.  For $w=(u,n)\in P_1$ set
\[
s_w^{(1)}:=(s_u^{(0)},0)\in [-L,L]^{d_0+1}.
\]
Then $\ell(w,s_w^{(1)})^\ast$ projects onto $\ell(u,s_u^{(0)})^\ast$ in the first $d_0$ coordinates, which is disjoint from $P_0$ (the projection of $P_1$ onto the first $d_0$ coordinates). Whence
$\ell(w,s_w^{(1)})^\ast\cap P_1=\emptyset$.

If $w=(u,n)\in ([N]^{d_0}\setminus P_0)\times [N]$, take instead
\[
s_w^{(1)}:=(0,\dots,0,1)\in [-L,L]^{d_0+1}.
\]
Then $\ell(w,s_w^{(1)})^\ast$ varies only in the last coordinate and never meets $P_1=P_0\times[N]$.

Finally set $P:=P_1\times [M]\subset [N]^{d_0+1}\times[M]$ and for $v=(w,m)$ define
\[
s_v:=(s_{w}^{(1)},1)\in [-L,L]^{d_0+1}\times\{1\}.
\]
By construction, $\ell(v,s_v)^\ast$ projects onto $\ell(w,s_{w}^{(1)})^*$ in the first $d_0+1$ coordinates, whence $\ell(v,s_v)^*\cap P=\emptyset$.
Moreover,
\[
|P|=|P_0|\cdot N\cdot M\ \gg_\varepsilon\ N^{d_0-\varepsilon}\cdot N\cdot N^{1-\varepsilon}
= N^{(d_0+1)+1-2\varepsilon}.
\]
Renaming $d:=d_0+1$ completes the proof.
\end{proof}

Finally, we combine Corollary~\ref{cor:lattice-nikodym} with Proposition~\ref{prp:higher dim to 2d} to obtain planar Nikodym sets.

\begin{proof}[Proof of Theorem~\ref{thm:d=2}]
Take $\varepsilon=1/3$.  By Corollary~\ref{cor:lattice-nikodym} there exists $d=O(1)$ such that for all $N\ge 1$ there is
a set
\[
P\subset [N]^d\times [N^{2/3}]
\quad\text{with}\quad
|P|\gg N^{d+1/3},
\]
and with escaping directions $s_v\in[-N^{1/3},N^{1/3}]^d\times\{1\}$ from every point of the ambient box.
Now let $q$ be a large prime and set $N:=\lfloor q^{1/d}/10\rfloor$. Then $(3N)^d<q$, so Proposition~\ref{prp:higher dim to 2d}
applies (with $M=N^{2/3}$ and $L=N^{1/3}$) and yields
\[
\Nik(2,q)\ \le\ q^2-|P|
\ \le\ q^2-\Omega\!\bigl(N^{d+1/3}\bigr)
\ =\ q^2-\Omega\!\bigl(q^{1+1/(3d)}\bigr).
\]
This proves the theorem.
\end{proof}

\subsection{Minimal blocking sets from Nikodym sets}
Recall that $\PG(2,q)$ is the projective plane over $\F_q$, and that $B(2,q)$ is the maximum cardinality of a minimal blocking set in $\PG(2,q)$. In this short subsection, we verify Proposition~\ref{prop:blocking}. 

To this end, we use a change in perspective. While a blocking set is a set of points $P$ which intersects every line in $\PG(2, q)$, taking the dual yields a set of lines $L$ where $\bigcup_{\ell\in L}\ell = \PG(2, q)$. We call such a set of lines a \textit{cover}, and say $L$ is a \textit{minimal cover} if every $L'\subsetneq L$ is not a cover. Then, an equivalent way to define $B(2, q)$ is as the maximum cardinality of a minimal cover of $\PG(2, q)$.

With this established, we can prove our reduction.
\begin{proof}[Proof of Proposition~\ref{prop:blocking}] 

Fix a Nikodym set $N\subset \F_q^2$ of size $\Nik(2,q)$. We write points of $\PG(2, q)$ in projective notation $[x:y:z]$, with $[x:y:z] = [\lambda x: \lambda y: \lambda z]$ for $\lambda \in \F_q^{\times}$. Recall the following basic facts about $\PG(2, q)$.
\begin{enumerate}
    \item $\PG(2, q)$ consists of the $q^2$ \emph{affine points} of the form $[x: y: 1]$ which we identify with $(x, y) \in \F_q^2$, and the $(q + 1)$ \emph{points at infinity} of the form $[x:y:0]$.
    \item The points at infinity lie on a single \emph{line at infinity} $\ell_\infty$, with equation $z = 0$.
\end{enumerate}
\noindent Henceforth, let us identify any point $(x,y)\in \F_q^2$ with the corresponding affine point $[x:y:1]$.

For each $v\in \F_q^2$, there exists some direction $z_v\in \F_q^2\setminus\{0\}$ so that $\ell(v,z_v)^*\subset N$. Set $\ell_v:=\ell(v,z_v)$. The union of $\{\ell_v:v\in \F_q^2\}$ contains all affine points of $\PG(2, q)$. Hence, the line family
$$L_0 = \{\ell_v:v\in \F_q^2\} \cup \{\ell_\infty\}$$
covers $\PG(2, q)$.

We observe that for each $v \in \F_q^2\backslash N$, the only element of $L_0$ that contains $v$ is $\ell_v$: indeed, for any $w \notin v$, all affine points of $\ell_w \backslash \{w\}$ lies in $N$, so $v \notin \ell_w$. The line at infinity $\ell_\infty$ consists of points at infinity, so it also cannot contain $v$.

Let $L \subset L_0$ be any minimal subfamily of $L_0$ that covers $\PG(2, q)$. In light of the preceding observation, we must have
$$L \supset \{\ell_v: v \in \F_q^2\backslash N\}.$$
Therefore, $L$ is a minimal cover with $\abs{L} \geq q^2 - \abs{N}$, as desired.
\end{proof}

\section{Minimal distance configurations} 
\label{sec:Heilbronn}
In this section, we prove our results on minimal distance problem. First, we show our main observation, Theorem~\ref{thm:lattice-to-PL}.

\begin{theorem}[Restate of Theorem~\ref{thm:lattice-to-PL}]
Fix $d\ge 1$ and let $N,M,L\ge 1$ be integers with $N\ge ML$.
Let $P\subset [N]^d\times [M]\subset \mathbb{Z}^{d+1}$ be a set of lattice points.
Assume that for each $p\in P$ we are given an integer direction vector
\[
  s_p=(u_p,1)\in \mathbb{Z}^{d+1}
  \qquad\text{with}\qquad
  \|u_p\|_\infty\le L,
\]
such that for all distinct $p,p'\in P$ one has
\[
  p'\notin p+\mathbb{R}s_p
\]
which, due to the last coordinate of $s_p$ being $1$, is equivalent to $p'\notin p+\mathbb{Z}s_p$. Then there exists a set of points $p_1, \cdots, p_{|P|} \in [0, 1]^{d + 1}$ and lines $\ell_1, \cdots, \ell_{|P|}$ with $p_i \in \ell_i$ such that for each $i \neq j$, we have
\[
  d_\infty(p_i, \ell_j)\ \ge\ \frac{1}{2N}.
\]
\end{theorem}

\begin{proof}
    Let $\varphi: [N]^d \times [M]\to [0,1]^{d+1}$ be the linear map $$\varphi(n_1,\dots,n_d,m) :=\left(\frac{1}{N} n_1, \cdots, \frac{1}{N} n_d, \frac{L}{N} m\right)$$
    where the image lies in $[0, 1]^{d + 1}$ by the assumption $LM \le N$. Let our point set be $Q = \{\varphi(p): p \in P\}$ in $[0, 1]^{d + 1}$, and let the line $\ell_q$ corresponding to $q = \varphi(p) \in Q$ be defined by the direction vector $\varphi(s_p)$.

    It suffices to check that, for distinct $q = \varphi(p), q' = \varphi(p')\in Q$ and $r = \varphi(p)+t\cdot \varphi(s_p)$ with $t \in \R$, we have
    $$\lVert{q' - r\rVert}_\infty\ge \frac{1}{2N}.$$ 
    This implies that $d_{\infty}(q', \ell_{q}) \ge 1 / 2N$, as desired. 

    We first suppose $\abs{p_{d+1}'-p_{d+1}-t\cdot (s_p)_{d+1}} > \frac{1}{2L}$. Then we have that 
    $$||\varphi(p')-\varphi(p)-t\cdot \varphi(s_{p})||_\infty\ge |\varphi(p')_{d+1}-\varphi(p)_{d+1}-t\cdot \varphi(s_p)_{d+1}|> \frac{L}{N}\cdot \frac{1}{2L}=\frac{1}{2N}.$$

    Now we suppose $\abs{p_{d+1}'-p_{d+1}-t\cdot (s_p)_{d+1}} \leq \frac{1}{2L}$. By definition, we have
    $(s_p)_{d+1} = 1$, so
    $$\abs{t - (p_{d+1}'-p_{d+1})} \leq \frac{1}{2L}.$$
    Set $\Tilde{t}:= p_{d+1}'-p_{d+1}$. Observe that 
    $$||\Tilde{t}\cdot \varphi(s_p)-t\cdot \varphi(s_p)||_\infty \le \frac{1}{2L}||\varphi(s_p)||_\infty = \frac{1}{2L} \cdot \frac{L}{N}< \frac{1}{2N}$$
    where we have $||\varphi(s_p)||_\infty \leq \frac{L}{N}$ by the assumptions that $s_p = (u_p, 1)$ and $||u_p||_\infty \leq L$. Then, we note that $p'-p-\Tilde{t}\cdot s_p \neq 0$ since $p'\not\in p+ \R s_p$, and each of the vectors $p', p, s_p, \Tilde{t}$ are integral. Hence, we must have
    $$|| p'-p-\Tilde{t}\cdot s_p ||_\infty \geq 1.$$
    Therefore, we have
    $$||\varphi(p') - \varphi(p) - \Tilde{t}\cdot \varphi(s_p)||_\infty \geq \frac{1}{N} || p'-p-\Tilde{t}\cdot s_p ||_\infty \geq \frac{1}{N}$$
    Finally, by the triangle inequality, we conclude that
    \[||\varphi(p') - \varphi(p) - t\cdot \varphi(s_p))||_\infty \ge ||\varphi(p') - \varphi(p) - \Tilde{t}\cdot \varphi(s_p)||_\infty -||(t-\Tilde{t})\varphi(s_p)||_\infty \geq \frac{1}{N}-\frac{1}{2N}=\frac{1}{2N}\]
    as desired.
\end{proof}
Corollaries~\ref{cor:PL2} and \ref{cor:PLd} follow by observing that the constructions in Theorems~\ref{thm:d=2} and \ref{thm:prime-high-dim} happen to satisfy the ``bounded slope criterion".
\begin{proof}[Proof of Corollary~\ref{cor:PL2}]
    Part 1) follows immediately from part 2) and Lemma~\ref{thm:Ruzsa-square}. So it suffices to establish 2).

    Recall the proof of Lemma~\ref{lem:Ruzsa-lift}: Let $q$ be any large prime. Let $N = \lfloor q / 3 \rfloor$ and $M = \lfloor \sqrt{q} / 2 \rfloor$. Let $A$ be the largest square--difference--free subset of $[\lfloor q / 10 \rfloor]$. We define the set 
    $$P = \{(x, y) \in [N] \times [M]: 2x - y^2 \in A\}$$
    and for each $p = (x,y) \in P$, define the associated line $\ell_p = \{(x, y) + t(y, 1): t\in \Z\}$ with slope $s_p = (y, 1)$. We showed that $\{(p,\ell_p):p\in P\}$ is an induced matching in the point--line incidence graph of $\F_q^2$. In particular, this implies that $p' \notin p + \R s_p$ for each distinct $p, p' \in P$. Hence, we can apply Theorem~\ref{thm:lattice-to-PL} with $L = M$, noting that $LM < N$ by definition. We conclude that there exists points $p_1, \cdots, p_n \in [0, 1]^2$ and lines $\ell_1, \cdots, \ell_p$ with $p_i \in \ell_j$ and for distinct $i \neq j$
    $$d_\infty(p_i, \ell_j) \leq \delta$$
    where $n = M |A|$ and $\delta = \frac{1}{2N}$.

    Thus assuming $\PL_2(0.5 + c)$ does hold, we must have
    $$M |A| \leq (2N)^{1.5 - c + o(1)}$$
    which given our choice of $M$ and $N$ gives
    $$|A| \leq N^{1 - c + o(1)}.$$
    Thus, any square--difference--free subset of $[\lfloor q / 10 \rfloor]$ has size at most $q^{1 - c + o(1)}$. Since this holds for any prime $q$, by Bertrand's postulate this also holds with $\lfloor q / 10 \rfloor$ replaced by any positive integer.
\end{proof}
\begin{proof}[Proof of Corollary~\ref{cor:PLd}]
    Given $\gamma >0$, let $d$ be the dimension given by Corollary~\ref{cor:lattice-nikodym} with $\varepsilon :=\gamma /3$ (thus in particular $d\le \exp(O(\gamma^{-1}))$). We will show that $\PL_{d+1}(\gamma)$ fails, meaning we can take $d_0(\gamma)=d+1$.

    By definition of $d$, given any $N$, we can find $P\subset [N]^{d}\times [N^{1-\varepsilon}]$ of size $|P|\gg_\varepsilon N^{d+1-2\varepsilon}$, and slopes $s_p\in [-N^{-\varepsilon},N^{\varepsilon}]^d\times \{1\}$ for $p\in P$ so that $\ell(p,s_p)^*\cap P=\emptyset$ (recall that this means that there is no $t\in \Z\setminus \{0\}$ so that $p+t\cdot s_p\in P$). In particular, we have that for distinct $p,p'\in P$, that $p'\not \in p+\Z s_p$.

    Thus, we can apply Theorem~\ref{thm:lattice-to-PL} to $P$ with parameters \[N:=N,M:=N^{1-\varepsilon},L:=N^{\varepsilon}.\] We conclude there exists a set of points $p_1,\dots,p_{|P|}\in [0,1]^{d+1}$ and lines $\ell_1,\dots,\ell_{|P|}$ with $p_i\in \ell_i$ such that for each $i\neq j$ we have\[d_\infty(p_i,\ell_j)\ge \frac{1}{2N}.\]
    Recalling $n:=|P|\gg_\varepsilon N^{d+1-2\varepsilon}$ and $2\varepsilon<\gamma$, it is impossible for the bound ``$n< (2N)^{d+1-\gamma+o(1)}$'' to hold, meaning $\PL_{d+1}(\gamma)$ is false, as desired.
\end{proof}

\section{Point--hyperplane matchings}\label{sec:higherdim}

The induced matching problem extends naturally beyond point-line incidences.  For example, for an integer $d\ge 2$, let
$\IMPH(d,q)$ denote the maximum size of an induced matching in the point--hyperplane incidence graph of $\F_q^d$: the left vertex set is $\F_q^d$
(points), the right vertex set is the set of affine hyperplanes in $\F_q^d$, and a point is adjacent
to a hyperplane if and only if it lies on it. 

The same eigenvalue method underlying \eqref{eq:vinh} gives a uniform upper bound for all $d$. 

\begin{proposition}[Point--hyperplane upper bound]\label{prop:point-hyperplane-upper}
  Let $q$ be a prime power and $d\ge 2$.  If $P=\{p_1,\dots,p_n\}\subset\F_q^d$ and
  $H=\{\pi_1,\dots,\pi_n\}$ is a family of affine hyperplanes in $\F_q^d$ such that
  \[ 
    p_i\in \pi_j\ \text{ holds if and only if }\ i=j, 
  \]
  then
  \[
    n\le q^{\frac{d+1}{2}}+q.
  \]
  Equivalently, $\IMPH(d,q)\le q^{\frac{d+1}{2}}+q$.
\end{proposition}

\begin{proof}
  Vinh \cite{Vinh2011ST} proved more generally that for any sets of points $P\subset\F_q^d$ and affine hyperplanes $H$ in $\F_q^d$,
  the number of incidences satisfies
  \[
    \Bigl|I(P,H)-\frac{|P||H|}{q}\Bigr|\le q^{\frac{d-1}{2}}\sqrt{|P||H|},
  \]
  see \cite{Vinh2011ST}.  In our situation $|P|=|H|=n$ and $I(P,H)=n$, so
  \[
    \bigl|n-n^2/q\bigr|\le q^{\frac{d-1}{2}}n.
  \]
  If $n\le q$ there is nothing to prove.  Otherwise $n^2/q-n\ge 0$, and the last inequality gives
  \[
    \frac{n^2}{q}-n\le q^{\frac{d-1}{2}}n,\qquad\text{i.e. }\qquad
    n\le q\bigl(q^{\frac{d-1}{2}}+1\bigr)=q^{\frac{d+1}{2}}+q.
  \]
\end{proof}

For $d=3$, Proposition~\ref{prop:point-hyperplane-upper} asserts $\IMPH(3,q)\le q^2+q$.
Perhaps surprisingly, in contrast with the planar case, we note that the correct scale in three dimensions is already achieved by a classical quadratic construction, for every prime or prime power $q$. 

Assume $\operatorname{char}(\F_q)\ne 2$ and fix $a\in\F_q^\times$ such that $-a$ is a non-square.
Consider the (elliptic) paraboloid
\[
  P\ :=\ \{(x,y,z)\in\F_q^3:\ x=y^2+a z^2\}\ =\ \{(y^2+a z^2,\,y,\,z):\ y,z\in\F_q\}.
\]
For a point $p=(x_0,y_0,z_0)\in P$ define the affine plane
\[
  \pi_p\ :=\ \Bigl\{(x,y,z)\in\F_q^3:\ x-2y_0y-2a z_0z= -y_0^2-a z_0^2\Bigr\}.
\]
The next proposition will show that $\pi_p$ is the tangent plane to $P$ at $p$, in the sense that $\pi_p \cap P = \left\{p\right\}$ for every $p \in P$. Since $|P|=q^2$, this example will thereby provide a lower bound construction for $\IMPH(3,q)$.

\begin{proposition}\label{prop:d3-construction}
  With notation as above, the family of point--plane pairs
  \[
    \{(p,\pi_p):\ p\in P\}
  \]
  is an induced matching in the point--plane incidence graph of $\F_q^3$.
  In particular, 
  $$\IMPH(3,q)\ge q^2.$$
\end{proposition}

\begin{proof}
  Fix $p=(x_0,y_0,z_0)\in P$.  Since $x_0=y_0^2+a z_0^2$, we have
  \[
    x_0-2y_0y_0-2a z_0z_0 = -y_0^2-a z_0^2
  \]
  and hence $p\in\pi_p$.

  Now let $p'=(x',y',z')\in P\cap\pi_p$.  Substituting $x'=y'^2+a z'^2$ into the plane equation gives
  \[
    y'^2+a z'^2-2y_0y'-2a z_0z'\ =\ -y_0^2-a z_0^2,
  \]
  or equivalently
  \[
    (y'-y_0)^2+a(z'-z_0)^2\ =\ 0.
  \]
  Since $-a$ is a non-square, the only solution is $y'=y_0$ and $z'=z_0$.  It follows that $p'=p$, so
  $\pi_p\cap P=\{p\}$.

  Therefore, no point in $P$ lies on the plane corresponding to a different point, and the set of
  pairs $\{(p,\pi_p):p\in P\}$ is an induced matching.
\end{proof}

\begin{remark}
  The construction in Proposition~\ref{prop:d3-construction} is an affine model of the classical
  Brown construction/elliptic quadric (ovoid) in $\mathrm{PG}(3,q)$; see, for instance, \cite[Chapter 4.3]{Guth}.
  Combined with Proposition~\ref{prop:point-hyperplane-upper}, it yields
  \[
    \IMPH(3,q)=q^2+O(q).
  \]

It is tempting to try to imitate the construction from Proposition~\ref{prop:d3-construction} in higher dimensions by hyperplanes which are tangent to a quadratic hypersurface.  The preceding proof, however, is quite specific to dimension $d=3$: the ``error term'' is controlled by a binary quadratic form $u^2+a v^2$ with no non--trivial zeros, which is only available in two variables. Determining the correct order of magnitude of $\IMPH(d,q)$ for $d\ge 4$ seems to be an interesting
finite-geometric question, related to the existence of large minimal blocking sets in higher dimensions \cite{MPS07}.
\end{remark}

\section{Concluding remarks}\label{sec:conclusion}

This paper develops several new lower bounds for induced matchings in point--line incidence graphs over finite fields.
We close by highlighting a few directions where we think the induced-matching viewpoint can keep paying dividends.

\subsection*{Simple things we don't know}

On the upper-bound side our understanding is still rather limited.
For $d=2$, Vinh's finite-field Szemer\'edi--Trotter theorem \cite{Vinh2011ST} implies that for every prime power $q$,
\begin{equation}\label{eq:vinh-concl}
  \IM(2,q)\ \le\ q^{3/2}+q.
\end{equation}
This $q^{3/2}$ exponent is sharp up to constants when $q$ is a square, thanks to the Hermitian unital.
Beyond this, however, essentially no nontrivial upper bounds are known for point--line induced matchings in $\F_q^d$ once $d\ge 3$.
In contrast to the point--hyperplane setting (where one still has strong spectral/incidence tools),
we do not currently know any interesting analogue of \eqref{eq:vinh-concl} for points and lines in higher dimensions. Even a qualitative asymptotic improvement over the trivial bound $\IM(d,q)\le q^d$ would be a major advance. 

\begin{conj}\label{conj:highd-oqd}
For every fixed $d\ge 3$,
\[
  \IM(d,q)\ =\ o(q^d)\qquad \text{as }q\to\infty.
\]
\end{conj}

In the planar case, the sharp examples for \eqref{eq:vinh-concl} are fundamentally extension-field phenomena.
Naturally, this motivates the following strengthening for induced matchings over a prime field. 

\begin{conj}\label{conj:prime-saving}
There exists an absolute constant $c>0$ such that for all sufficiently large \emph{primes} $q$,
\[
  \IM(2,q)\ \le\ q^{3/2-c}.
\]
\end{conj}

Morally, Conjecture~\ref{conj:prime-saving} asserts that one cannot realize a ``unitary-like'' configuration over a prime field which would support induced matchings on the $q^{3/2}$ scale.  Two rather remarkable consequences of such a prime-field saving are worth recording.

\subsection*{Cliques in Paley graphs}

Assume $q\equiv 1\pmod 4$ is prime and let $G_q$ be the Paley graph on $\F_q$.
As observed in Proposition~\ref{prop:paley-to-im}, Paley cliques lift to induced matchings:
\begin{equation}\label{eq:paley-im-ineq-concl}
  q\,\omega(G_q)\ \le\ \IM(2,q).
\end{equation}
Consequently, Conjecture~\ref{conj:prime-saving} would immediately yield a matching power saving for Paley clique numbers,
with the \emph{same} exponent:
\[
  \IM(2,q)\ \ll\ q^{3/2-c}
  \qquad\Longrightarrow\qquad
  \omega(G_q)\ \ll\ q^{1/2-c}.
\]
Breaking the square-root barrier for Paley graphs of prime order by a polynomial factor is a classical open problem
at the intersection of additive combinatorics, pseudorandomness, and analytic number theory; see, e.g.,
\cite{Yip2022PrimePowerPaley,kunisky2023} and the references therein.
The strongest unconditional improvement in the prime case remains only a constant-factor gain due to Hanson--Petridis \cite{HansonPetridis2021}.

\subsection*{The Furstenberg--S\'ark\"ozy problem}

Our Ruzsa set lifting (Proposition~\ref{lem:Ruzsa-lift} and its prime-power variants) shows that induced matchings
interact naturally with the Furstenberg--S\'ark\"ozy problem.
Let
\[
  s(N)\ :=\ \max\Bigl\{|A|:\ A\subset [N],\ (A-A)\cap\{m^2:\ m\in\Z\setminus\{0\}\}=\emptyset\Bigr\}
\]
denote the largest size of a subset of $[N]=\{1,\dots,N\}$ with no \emph{nonzero} square difference.
The Furstenberg--S\'ark\"ozy theorem asks about the asymptotic behavior of $s(N)$, as a function of $N$.
The current record for the upper bound is due to Green--Sawhney \cite{GreenSawhney}:
\begin{equation}\label{eq:GS-record-concl}
  s(N)\ \ll\ N\exp\!\bigl(-c\sqrt{\log N}\bigr).
\end{equation}
A well-known folklore goal is to strengthen this to a \emph{power saving}:
\begin{equation}\label{eq:FS-folklore-concl}
  s(N)\ \ll\ N^{1-c}
\end{equation}
for some absolute $c>0$.

Conjecture~\ref{conj:prime-saving} would imply such a power saving (and with the same exponent).
Indeed, Proposition~\ref{lem:Ruzsa-lift} shows that if $A\subset [\lfloor q/10\rfloor]$ is square-difference-free, then
\[
  \IM(2,q)\ \gg\ q^{1/2}|A|.
\]
Therefore, if $\IM(2,q)\ll q^{3/2-c}$ holds for primes $q$, we obtain
\[
  |A|q^{1/2}\ \ll\ \IM(2,q)\ \ll\ q^{3/2-c}
  \qquad\Longrightarrow\qquad
  |A|\ \ll\ q^{1-c}\ \asymp\ N^{1-c},
\]
which is exactly \eqref{eq:FS-folklore-concl}.
Thus Conjecture~\ref{conj:prime-saving} would provide a rather unexpected route to the long-sought power-saving regime
for Furstenberg--S\'ark\"ozy.

\subsection*{Euclidean point--line separation and Heilbronn}

Induced matchings are also a finite-field model for Euclidean point--line separation questions.
Given $n$ pairs $\{p_i\in \ell_i\}_{i=1}^n$ in the unit square, define
\[
  \delta := \min_{i\neq j} d(p_i,\ell_j),
\]
where $d(\cdot,\cdot)$ denotes Euclidean distance.
Cohen--Pohoata--Zakharov \cite{CPZ} proved that one always has $\delta\lesssim n^{-2/3+o(1)}$,
and this implies the upper bound $\Delta(n)\lesssim n^{-7/6+o(1)}$ for Heilbronn's triangle problem.

A natural open problem is to improve the $2/3$ exponent in the minimal-distance problem to $2/3+c$.
Such an improvement would generate a further new record of $\Delta(n) \lesssim n^{-7/6-c+o(1)}$ for the Heilbronn triangle problem. In Corollary \ref{cor:PL2}, we also observed that improvements for the minimal distance problem have nontrivial implications for square-difference-free sets (via a geometric encoding of the type used in Proposition~\ref{lem:Ruzsa-lift}). We view this as further evidence that induced matchings provide a useful organizing principle for the intersection of finite-field incidence combinatorics, Euclidean separation, and additive number theory.

\subsection*{Ramsey theory and forbidden configurations}

Finite-geometric incidence structures have long supplied extremal and Ramsey constructions via their incidence graphs. A striking recent example is the work of Mattheus--Verstra\"ete \cite{MV} on $R(4,t)$:
they exploit the Hermitian unital and its ``no O'Nan configuration'' property to build very dense graphs
avoiding the one-subdivision of $K_4$, leading to a nearly sharp lower bound for $R(4,t)$.

Our norm-hypersurface construction from Theorem~\ref{thm:small-power} can be viewed as higher-degree analogues of unitary geometry: we obtain large point sets $P$ lying on an algebraic hypersurface together with a distinguished line through each $p\in P$ meeting $P$ only at $p$.  It would be very interesting to understand whether these higher-$k$ objects also satisfy local forbidden-configuration phenomena in the spirit of the O'Nan property, and whether such structure could be leveraged to build new extremal and Ramsey constructions.

\medskip
\noindent {\bf{Acknowledgments.}} ZH was supported by SNSF grant 200021-228014. CP was supported by NSF grant DMS-2246659. JV was supported by the NSF FRG Award DMS-1952786 and NSF Award DMS-2347832. The authors would like to thank Vaughan McDonald for valuable discussions about the manuscript.

\appendix
\section{Polynomial identities via Lagrange interpolation}\label{app:binomial-alternating-sum}

In this appendix, we record the proof of Proposition \ref{prop:shifted-kth-power-sum} that we used in Section \ref{sec:Ruzsa-prime-powers} in the proof of Theorem \ref{thm:prime-power-high-dim}. 

\begin{proposition}\label{prop:shifted-kth-power-sum}
Let $k\ge 1$ and $f(x)\in \mathbb{Z}[x]$. Then
\begin{equation}\label{eq:main-identity}
\sum_{i=0}^{k-1}(-1)^i\binom{k-1}{i}\,\bigl(f(x)-i\bigr)^k
= k!\,f(x)-\frac{k!(k-1)}{2}.
\end{equation}
Equivalently,
\[
\sum_{i = 0}^{k - 1} (-1)^i \binom{k - 1}{i} \left(2f + k-1 - 2i\right)^k = k! \cdot 2^{k} \cdot f.
\]
\end{proposition}

\begin{proof}
We use the following fairly well-known polynomial identity.

\begin{lemma}\label{lem:boole}
Let $n\ge 0$, and let
\[
p(t)=a_0 t^{n}+a_1 t^{n-1}+\cdots + a_{n-1}t + a_n
\]
be a polynomial with coefficients in $\mathbb{Q}[x]$. Then, for all $a,b$ with $b\neq 0$, we have that
\begin{equation}\label{eq:pohoata}
\sum_{j=0}^{n} (-1)^{\,n-j}\binom{n}{j}\,p(a+j b)=a_0\,b^{n}\,n!.
\end{equation}
\end{lemma}

Lemma \ref{lem:boole} is typically attributed to Boole and can be proved, for example, using Lagrange interpolation formula. See \cite{Pohoata2008} and the references therein. To derive Proposition \ref{prop:shifted-kth-power-sum} from Lemma \ref{lem:boole}, we will work in the coefficient ring $\mathbb{Q}[x]$, and set $c:=\dfrac{k-1}{2}$.
Consider the polynomial in $t$,
\[
P(t):=\bigl(f(x)-t\bigr)^k-\bigl(c-t\bigr)^k \in \mathbb{Q}[x][t].
\]
As a polynomial in $t$, both $(f(x)-t)^k$ and $(c-t)^k$ have leading term $(-t)^k$,
so the $t^k$ terms cancel and hence $\deg_t P\le k-1$.

We now compute the coefficient of $t^{k-1}$ in $P(t)$. Using the binomial theorem,
\[
(f(x)-t)^k=\sum_{r=0}^{k}\binom{k}{r} f(x)^{k-r}(-t)^r,
\]
so the $t^{k-1}$-term is $\binom{k}{k-1}f(x)(-t)^{k-1}=k f(x)(-1)^{k-1}t^{k-1}$.
Similarly, the $t^{k-1}$-term of $(c-t)^k$ is $k c(-1)^{k-1}t^{k-1}$.
Therefore the leading coefficient of $P(t)$ (as a degree $\le k-1$ polynomial in $t$) is
\[
a_0=(-1)^{k-1}k\bigl(f(x)-c\bigr).
\]

Apply Lemma~\ref{lem:boole} to $P(t)$ with $n=k-1$, $a=0$, $b=1$ to get:
\[
\sum_{i=0}^{k-1}(-1)^{\,k-1-i}\binom{k-1}{i}\,P(i)
=a_0\cdot 1^{k-1}\cdot (k-1)!
=(-1)^{k-1}k\bigl(f(x)-c\bigr)(k-1)!.
\]
Multiplying both sides by $(-1)^{k-1}$ yields
\begin{equation}\label{eq:after-pohoata}
\sum_{i=0}^{k-1}(-1)^{i}\binom{k-1}{i}\,P(i)
= k!\bigl(f(x)-c\bigr).
\end{equation}
Expanding $P(i)=(f(x)-i)^k-(c-i)^k$ in~\eqref{eq:after-pohoata} gives
\begin{equation}\label{eq:split}
\sum_{i=0}^{k-1}(-1)^{i}\binom{k-1}{i}\,(f(x)-i)^k
\;-\;
\sum_{i=0}^{k-1}(-1)^{i}\binom{k-1}{i}\,(c-i)^k
= k!\bigl(f(x)-c\bigr).
\end{equation}

It remains to show that the second sum in~\eqref{eq:split} is $0$. Let
\[
T:=\sum_{i=0}^{k-1}(-1)^{i}\binom{k-1}{i}\,(c-i)^k.
\]
Pair the summand with index $i$ with the summand at index $k-1-i$. Since
$\binom{k-1}{k-1-i}=\binom{k-1}{i}$ and $
c-(k-1-i)=\frac{k-1}{2}-(k-1-i)=-(c-i)$, we obtain
\[
(c-(k-1-i))^k = \bigl(-(c-i)\bigr)^k = (-1)^k(c-i)^k.
\]
Moreover, $(-1)^{k-1-i}=(-1)^{k-1}(-1)^i$. It follows that the paired summand equals
\[
(-1)^{k-1-i}\binom{k-1}{k-1-i}(c-(k-1-i))^k
=
-\,(-1)^i\binom{k-1}{i}(c-i)^k,
\]
Thus each pair cancels, and if $k$ is odd the
unique fixed point $i=c=(k-1)/2$ contributes $(c-c)^k=0$. Therefore $T=0$.

With $T=0$, equation~\eqref{eq:split} becomes
\[
\sum_{i=0}^{k-1}(-1)^{i}\binom{k-1}{i}\,(f(x)-i)^k
= k!\bigl(f(x)-c\bigr)
= k!\,f(x)-k!\,\frac{k-1}{2},
\]
which is exactly~\eqref{eq:main-identity}.
\end{proof}

\end{document}